\documentclass{amsart}
\usepackage{amsmath}
\usepackage{amssymb}
\usepackage{amsfonts}
\usepackage{graphicx}
\usepackage{algorithm}
\usepackage{algorithmic}

\newtheorem{thm}{Theorem}
\newtheorem{corollary}{Corollary}
\newtheorem{lemma}{Lemma}
\newtheorem{remark}{Remark}

\newcommand{\reals}{\mathbb{R}}

\title{Precise Matching of PL Curves in $\reals^N$ in the Square Root Velocity Framework}
\author{Sayani Lahiri}
\author{Daniel Robinson}
\author{Eric Klassen}
\begin{document}
\maketitle

\section{Introduction}

In \cite{ElasticPAMI} and other related papers, Srivastava et al introduced a new method for analyzing the shapes of  absolutely continuous functions $[0,1]\to \reals^N$. By ``shape" we mean that under this analysis, two such functions are considered equivalent if they only differ by a reparametrization, i.e., by composition with a diffeomorphism $[0,1]\to[0,1]$. The method is based on producing a bijection between the set of absolutely continuous functions (starting at the origin) and $L^2(I,\reals^N)$, and then transferring the $L^2$ metric back to the set of absolutely continuous functions. The $L^2$ function corresponding to a given absolutely continuous function is called its {\it square root velocity function} (SRVF), and the general method is referred to by the same name. The result is a complete metric on the space of absolutely continuous functions starting at the origin. Furthermore, with respect to this metric, the group of diffeomorphisms acts by isometries. This makes it possible to mod out $L^2(I,\reals^N)$ by an appropriate group of reparametrizations, resulting in a quotient space that is also a complete metric space. 

The SRVF metric and the corresponding quotient construction have proved quite useful for analyzing shapes of functions and curves for several reasons: 
\begin{itemize}
\item The metric has a compelling geometric interpretation as an elastic metric (see \cite{ElasticPAMI}), under which optimal deformations minimize a combination of bending and stretching. 
\item It provides a very effective solution to the classical problem of aligning two functions $\reals\to\reals$ by warping their domains. (See Tucker et al \cite{TuckerWuSrivastava}.) 
\item It can easily be adapted to a method of comparing closed curves, which comprise a complete subspace of the metric space of all curves. These closed curves are especially important because they occur as outlines of images. (See  \cite{ElasticPAMI}.)
\item With some modifications, it can effectively be adapted to the analysis of curves up to affine transformation. (See Bryner et al \cite{Bryner2DAffine}.)
\end{itemize}

A fundamental problem that arises in the implementation of this method is the ``optimal matching" problem: Given two functions $I\to\reals^N$, determine reparametrizations of these functions that achieve the infimum of the distance between the two corresponding orbits under the reparametrization group. Finding such an optimal matching is important not only because it results in a precise computation of the distance between two orbits, but also because it allows one to find shortest geodesics in the quotient space. In theory, we do not know whether such a pair of optimal reparametrizations always exists! In most previous implementations, a solution to this optimal matching problem has been approximated using a dynamic programming algorithm (once again, see \cite{ElasticPAMI}) . 

The current paper has two primary goals: (1) to establish the theoretical underpinnings of the SRVF method, especially the delicate quotient construction referred to above and (2) to exhibit an algorithm that provides a precise solution to the optimal matching problem for continuous piecewise linear functions. The set of piecewise linear functions is very useful because it is the simplest way of interpolating functions for which we have only a finite set of data points, and because it is dense in the space of absolutely continuous functions with respect to the SRVF metric.

There is considerable literature on similar methods for analyzing curves. For example in Younes et al \cite{YounesMichor} a representation of planar curves is used that is similar to SRVF, but involves the complex square root of the velocity (as opposed to the SRVF method, which only takes the square root of the magnitude of the velocity). This method results in a beautiful way of handling closed curves, but does not generalize easily to curves in $\reals^N$. Also, in \cite{YounesMichor}, only smooth curves are considered, which means that the resulting quotient space is not a complete metric space. Sundaramoorthi et al consider a similar metric on the space of smooth planar curves in \cite{SundaramoorthiMennucci}.

In Bauer et al \cite{BauerRTransform}, a whole family of metrics on planar curves is considered, which includes both the SRVF metric and the metric in \cite{YounesMichor} as special cases. However, this paper also does not generalize to curves in $\reals^N$. 

In Daniel Robinson's unpublished doctoral dissertation \cite{RobinsonDissertation}, a precise matching algorithm is introduced for PL functions $I\to\reals$, but it does not easily generalize to PL functions $I\to\reals^N$. Some of the theoretical material from Section 2 is also adapted from \cite{RobinsonDissertation}.

The main advances in the current paper are as follows: 
\begin{itemize}

\item A rigorous development of the SRVF metric.
\item A careful development of the quotient of $L^2(I,\reals^N)$ by the group of reparametrizations; this includes a characterization of the closed orbits involved in the construction.
\item A description and inplementation of an algorithm that gives a precise solution to the matching problem for PL curves, a class of curves that is dense in the space of all absolutely continuous curves.
\end{itemize}

One issue that this paper does {\it not} address, is the action of the group of rotations, $O(N,\reals)$, on the space of absolutely continuous curves. This part of the theory is easier because it involves a linear action by a compact finite dimensional Lie group, and there are straightforward analytic methods for handling it (see Srivastava et al \cite{ElasticPAMI} for details on how to do this).

The contents of this paper are as follows: In Section 2, we define the square root velocity function (SRVF) of an absolutely continuous function $I\to\reals^N$, and we define the group $\Gamma$ of reparametrizations; we also define a semigroup $\tilde\Gamma$ that contains $\Gamma$. In Sections 3 and 4 we prove that the closure of each orbit under $\Gamma$ can be expressed as an orbit under $\tilde\Gamma$. This is important since, if we wish our quotient space to inherit a metric, the orbits must be closed sets. (In Section 3, this theorem is proved for functions $I\to\reals$, while in Section 4 it is generalized to functions $I\to\reals^N$.) In Section 5, we begin to focus on piecewise linear functions, which comprise a dense subset of the set of all absolutely continuous functions with respect to the SRVF metric. In particular, we prove that if we are given two orbits $[q_1]$ and $[q_2]$ under the action of $\tilde\Gamma$, and if at least one of these orbits contains  the SRVF of a piecewise linear function,  then there exist  orbit representatives that realize the minimum distance between these orbits. (Such a pair of orbit representatives is called an {\it optimal matching} of the two orbits.) We also prove that if both of these orbits contain the SRVFs of PL functions, then this optimal pair of orbit representatives can be chosen to be the SRVFs of PL functions. In Section 6, we begin our discussion of how to produce an optimal matching between PL functions, setting up some basic terminology. In Section 7, we prove a theorem establishing certain properties that an optimal matching between PL functions must have. In Section 8, we give a precise algorithm for producing an optimal matching between two PL functions, based on the theorem proved in Section 7. Section 9 gives a few examples of optimal matchings produced by the algorithm described in Section 8.

We thank our colleague Dan Oberlin for several helpful conversations.

\section{Basic Quotient Construction for Curves in $\reals^N$}
In this paper, we consider absolutely continuous functions $I\to \reals^N$, where $I=[0,1]$. A function $f:[a,b]\to\reals$ is {\it absolutely continuous} if and only if it has a derivative $f'$ almost everywhere, $f'$ is Lebesgue integrable, and for all $t\in I$, $f(t)=f(0)+\int_0^t\,f'(u)\,du$.  (This is not the usual definition of absolute continuity, but it is well known to be equivalent to the usual definition; see, for example, Theorems 11, p. 125 and Theorem 14, p. 126  of \cite{RoydenFitzpatrick}.) Let $AC_0(I,\reals^N)$ denote the set of absolutely continuous functions $I\to \reals^N$ with the property that $f(0)=0$. We want to compare these functions up to reparameterization. In other words, given $f$ and $g$ in $AC_0(I,\reals^N)$, we want to consider them as equivalent if there exists a ``nice" homeomorphism $\gamma:I\to I$ such that $f\circ\gamma=g$. If they are not equivalent, we would like a quantitative measure of how far from being equivalent they are. Let $\Gamma$ denote the group of functions $\gamma:I\to I$ which have the following three properties: (1) $\gamma$ is absolutely continuous, (2) $\gamma(0)=0$ and $\gamma(1)=1$,  and (3) $\gamma'(t)>0$ almost everywhere. $\Gamma$ is a group under composition. Clearly, $\Gamma$ acts on $AC_0(I,\reals^N)$ from the right by composition. We would like to make the quotient set $AC_0(I,\reals^N)/\Gamma$ into a metric space in a reasonable way. There are two important issues to overcome here. The first is that to get a reasonable metric on a quotient space, it helps if the group acts by isometries. The second is that the orbits should be closed sets. We tackle these one at a time.

Before we turn to these two issues, it will be helpful to define a semigroup containing $\Gamma$. Let $\tilde\Gamma$ be the set of functions $\gamma:I\to I$ satisfying (1) $\gamma$ is absolutely continuous, (2) $\gamma(0)=0$ and $\gamma(1)=1$, and (3) $\gamma'(t)\geq 0$ almost everywhere. Note that $\tilde\Gamma$ is a semigroup, and also acts on $AC_0(I,\reals^N)$ from the right by composition.

We now describe a way to understand the action of $\tilde\Gamma$ (and, therefore, $\Gamma$) as an action by isometries. To do this, begin by defining a function $V:\reals^N\to \reals^N$ by 
\begin{equation*}
V(x)=
\begin{cases}
\frac{x}{\sqrt{|x|}}& \text{for $x\neq0$}
\\
0& \text{for $x=0$}
\end{cases}
\end{equation*}

Denote by $L^2(I,\reals^N)$ the space of square integrable functions $I\to \reals^N$, with standard $L^2$ inner product denoted by $\langle q_1,q_2\rangle$ and distance function defined by $d(q_1,q_2)=\sqrt{\langle q_1-q_2,q_1-q_2 \rangle}$. Define a function $Q:AC_0(I,\reals^N)\to L^2(I,\reals^N)$ by $Q(f)=V\circ f'$. It's easy to see that $Q$ is bijective (this is proved in \cite{RobinsonDissertation}); in fact given a function $q\in L^2(I)$, we can define $f(t)=\int_0^t\,q(u)|q(u)|\,du$, and then verify that $Q(f)=q$.

Since $\tilde\Gamma$ acts on $AC_0(I,\reals^N)$ and $Q:AC_0(I,\reals^N)\to L^2(I,\reals^N)$ is bijective, we can define an action of $\tilde\Gamma$ on $L^2(I,\reals^N)$ in a unique way to make $Q$ equivariant. In fact, it is easy to verify that the corresponding right action of $\tilde\Gamma$ on $L^2(I,\reals^N)$ is given by $(q*\gamma)(t)=q(\gamma(t))\sqrt{\gamma'(t)}$. Furthermore, this action of $\tilde\Gamma$ on $L^2(I,\reals^N)$ is by isometries since
$$\langle q_1*\gamma,q_2*\gamma\rangle=\int_0^1 q_1(\gamma(t))\sqrt{\gamma'(t)}q_2(\gamma(t))\sqrt{\gamma'(t)}\,dt$$
$$ = \int_0^1 q_1(\gamma(t))q_2(\gamma(t))\gamma'(t)\,dt=\int_0^1q_1(u)q_2(u)\,du=\langle q_1,q_2\rangle.
$$

Note that for the second to last equality, we relied on integration by substitution, which is valid because $\gamma$ is absolutely continuous. This is one important reason for insisting that our reparameterization functions are absolutely continuous. Thus, we replace our study of the action of $\tilde\Gamma$ on $AC_0(I,\reals^N)$ by the study of the corresponding action of $\tilde\Gamma$ on $L^2(I,\reals^N)$, which is an action by linear isometries. 
In what follows we will be interested both in the action of $\tilde\Gamma$, and in the restricted action of $\Gamma$. Note to the reader: Our definition of ``action by isometries" is simply that for all $\gamma\in\tilde\Gamma$ and for all $q_1,q_2\in L^2(I,\reals^N)$, $\langle q_1*\gamma,q_2*\gamma\rangle=\langle q_1,q_2\rangle$. While this equation implies that the map $L^2(I,\reals^N)\to L^2(I,\reals^N)$ induced by each $\gamma$ is injective, it does not imply that it is surjective. For example, suppose that $\gamma\in\tilde\Gamma$ is constant on some subinterval of $I$. Then for all $q\in L^2(I,\reals^N)$, $q*\gamma=0$ on this same subinterval. Of course, for $\gamma\in\Gamma$, the induced map is bijective, since $\Gamma$ is a group. 

Denote by $U(I,\reals^N)$ the unit sphere $\{q\in L^2(I,\reals^N):\int_0^1|q(t)|^2\,dt=1\}$. This corresponds to the set of functions in $AC_0(I,\reals^N)$ having arc length 1, since if $Q(f)=q$, it follows that $|f'(t)|=|q(t)|^2$, and the arclength of $f$ can be written as $\int_0^1|f'(t)|\,dt$. $U(I,\reals^N)$ is an invariant subset of $L^2(I,\reals^N)$ under the action of $\tilde\Gamma$, as is the sphere of any radius centered at $0$ in $L^2(I,\reals^N)$. If we wish to compare two curves in a way that is invariant to rescaling, a natural way to do this is to rescale both of them to have unit length before comparing them. Hence, we sometimes concentrate on the action of $\tilde\Gamma$ on $U(I,\reals^N)$. 

$U(I,\reals^N)$ is an infinite dimensional submanifold of $L^2(I,\reals^N)$. If we think of it as a Riemannian manifold, using the $L^2$-inner product as a Riemannian metric, then the geodesics are the arcs of great circles, where by ``great circle" we mean the intersection of $U(I,\reals^N)$ with any 2-dimensional linear subspace of $L^2(I,\reals^N)$. The corresponding (geodesic) distance function between any $q_1$ and $q_2$ in $U(I,\reals^N)$ is given by $\cos^{-1}(\langle q_1,q_2\rangle)$. Note that $\Gamma$ and $\tilde\Gamma$ act on $U(I,\reals^N)$ by isometries. 

We now return our attention to the action of $\Gamma$ on $L^2(I,\reals^N)$. Given $q\in L^2(I,\reals^N)$, let $q\Gamma$ denote the orbit of $q$ under $\Gamma$, and let $L^2(I,\reals^N)/\Gamma$ denote the set of all these orbits. Define a function $\rho:(L^2(I,\reals^N)/\Gamma)\times (L^2(I,\reals^N)/\Gamma)\to \reals$ by $\rho(q_1\Gamma,q_2\Gamma)=\inf_{(\gamma_1,\gamma_2)\in\Gamma\times\Gamma}d(q_1*\gamma_1,q_2*\gamma_2)=\inf_{\gamma\in\Gamma}d(q_1,q_2*\gamma)$. The last equality follows from the fact that $\Gamma$ acts by isometries. As usual, it's easy to show that $\rho$ is symmetric, satisfies the triangle inequality and is non-negative. However, it's also easy to find examples where $q_1\Gamma\neq q_2\Gamma$, but $\rho(q_1\Gamma,q_2\Gamma)=0$. The reason for this is that the orbits are not closed sets, so all you have to do is choose $q_2$ to be in the $L^2$-closure of $q_1\Gamma$, but not in the orbit itself, in order to create such an example.

For example, define $\tilde\gamma\in\tilde\Gamma$ by 
\begin{equation*}
\tilde\gamma(t)=
\begin{cases}
2t& \text{for $t<.5$}
\\
1& \text{for $t\geq .5$}
\end{cases}
\end{equation*}

While $\tilde\gamma\notin\Gamma$, we now construct a sequence $\{\gamma_n\}$ in $\Gamma$ such that $\{\sqrt{\gamma_n'}\}$ approaches $\sqrt{\tilde\gamma'}$ in the $L^2$ sense. To do this, let
\begin{equation*}
\gamma_n(t)=
\begin{cases}
(2-\frac{1}{n})t& \text{for $0\leq t\leq .5$}
\\
(1-\frac{1}{n})+\frac{1}{n}t& \text{for $.5<t\leq 1$}
\end{cases}
\end{equation*}

Now, let $q_0(t)\equiv c$ denote a constant function, where $c\in \reals^N$ is any nonzero vector. For each $\gamma\in\Gamma$, $(q_0*\gamma)(t)=\sqrt{\gamma'(t)}c$. Then $\sqrt{\tilde\gamma'}c\notin q_0\Gamma$, but $\sqrt{\tilde\gamma'}c$ is a limit point of $q_0\Gamma$ because each $\sqrt{\gamma_n'}c$ is in the orbit $q_0\Gamma$, and clearly $\sqrt{\gamma_n'}c\to \sqrt{\tilde\gamma'}c$ with respect to the $L^2$ norm. Hence $q_0*\tilde\gamma\not\in q_0\Gamma$, but $q_0*\tilde\gamma$ is in the closure of $q_0\Gamma$.

\begin{lemma} Assume $q_1$ and $q_2$ are elements of $L^2(I,\reals^N)$. Then $\rho(q_1\Gamma,q_2\Gamma)=0$ if and only if $q_1\Gamma$  and $q_2\Gamma$ have the same closure in $L^2(I,\reals^N)$.
\end{lemma}
\begin{proof} Suppose  $\rho(q_1\Gamma,q_2\Gamma)=0$. Then there exist sequences $\{\gamma_n\}$ and  $\{\tilde\gamma_n\}$ in $\Gamma$ such that $\lim_{n\to\infty} d(q_1*\gamma_n,q_2*\tilde\gamma_n)=0$. Because $\Gamma$ acts by isometries, it follows that $\lim_{n\to\infty} d(q_1,q_2*\tilde\gamma_n\gamma_n^{-1})=0$, proving $q_1$ is in the closure of $q_2\Gamma$. Since there was nothing special about the orbit representatives we chose, and the argument is symmetric, if follows that each orbit is in the closure of the other; hence, $q_1\Gamma$ and $q_2\Gamma$ have the same closure. The other direction is obvious and we omit it. 
\end{proof}

Because of this lemma, the closure of an orbit is a union of orbits; if we define a binary relation on $L^2(I,\reals^N)$ by stipulating that $q_1\sim q_2$ means that $q_1\Gamma$ and $q_2\Gamma$ have the same closure, then $\sim$  is an equivalence relation. Let  ${\mathcal S}(I,\reals^N)=L^2(I,\reals^N)/\sim$. Henceforth, we will use the symbol $[q]$ to denote the point in ${\mathcal S}(I,\reals^N)$ corresponding to the closure of the orbit $q\Gamma$. We will loosely refer to this as the ``orbit" of $q$, even though it is actually a closed-up orbit.  It is easily verified that our distance function $d$ on $L^2(I,\reals^N)$ induces a metric, which we also call $d$, on ${\mathcal S}(I,\reals^N)$, defined by 
$$d([q_1],[q_2])=\inf_{w_1\in[q_1],w_2\in[q_2]} d(w_1,w_2)$$
and this metric induces the quotient topology.  

We now prove a theorem giving a general form for elements of $L^2(I,\reals^N)$.

\begin{thm}\label{arclengthpar} Let $q\in L^2(I, \reals^N)$. Then $q$ can be written in a unique way as $q=w*\gamma$ where $\gamma\in\tilde\Gamma$ and $w\in L^2(I,\reals^N)$ has the property that $|w(t)|$ is constant a.e. for $t\in I$.
\end{thm}

\begin{proof}
By Theorem 4.1 of \cite{Stein-Shakarchi-Real-Analysis-book}, every absolutely continuous function on a closed interval is rectifiable, and by Theorem 4.3 of the same book, every rectifiable function has a constant speed parametrization, with the reparametrizing function being absolutely continuous. Since for every $q\in L^2(I, \reals^N)$, there is an absolutely continuous function $f$ such that $q=Q(f)$, our lemma follows immediately.
\end{proof}

We now wish to focus on a particular subset of ${\mathcal S}(I,\reals^N)$. We define a function $q:I\to \reals^N$ to be a {\em step function} if we can express $I$ as a finite disjoint union of subintervals on each of which $q$ is constant. Since we will only use this concept for $L^2$ functions, we don't care what happens at the endpoints of each subinterval. It is a well-known fact (Prop. 10, p.151 of \cite{RoydenFitzpatrick}) that the step functions are dense in $L^2(I,\reals^N)$. We define ${\mathcal S}(I,\reals^N)_{st}\subset{\mathcal S}(I,\reals^N)$ to be the set of all equivalence classes that contain at least one step function. We see immediately that ${\mathcal S}(I,\reals^N)_{st}$ is dense in ${\mathcal S}(I,\reals^N)$. Furthermore, under our bijection $Q$ between $AC_0$ and $L^2$, the step functions in $L^2$ correspond to the piecewise linear functions (with finitely many pieces) in $AC_0$. 

Later in this paper, we will prove that given two elements $[q]$ and $[w]$ of ${\mathcal S}(I,\reals^N)$, if at least one of them is in  ${\mathcal S}(I,\reals^N)_{st}$, then there exist elements $\tilde q\in[q]$ and $\tilde w\in[w]$ such that $d(\tilde q,\tilde w)=d([q],[w])$. We will also prove that if both $[q]$ and $[w]$ are in ${\mathcal S}(I,\reals^N)_{st}$, then $\tilde q$ and $\tilde w$ can both be chosen to be step functions. As a result, it will follow that ${\mathcal S}(I,\reals^N)_{st}$ is a geodesically convex subset of ${\mathcal S}(I,\reals^N)$, in the sense that given $[q]$ and $[w]$ in ${\mathcal S}(I,\reals^N)_{st}$, there exists a minimal geodesic in ${\mathcal S}(I,\reals^N)$ joining $[q]$ to $[w]$ that lies entirely in ${\mathcal S}(I,\reals^N)_{st}$.

\section{Orbit Structure in ${\mathcal S}(I,\reals)$}

In this section, we will specialize to the case $N=1$, so all of our functions are scalar valued instead of vector valued. In this case, $\Gamma\subset AC_0(I,\reals)$, and $Q(\Gamma)\subset U(I,\reals)$. In fact, $Q(\Gamma)=\{q\in U(I,\reals):q(t)>0 \hbox{ a.e.}\}$. To see this, we refer to Exercise 3.21 on page 82 of \cite{Leoni-Sobolev-book}, which states that a continuous, strictly increasing function $u:[a,b]\to \reals$ has an absolutely continuous inverse if and only if the set on which its derivative vanishes has measure 0. It is then immediate that the closure of $Q(\Gamma)$ in $U(I,\reals)$ is $\{q\in U(I,\reals):q(t)\geq 0 \hbox{ a.e.}\}$. Note that this last set is equal to $Q(\tilde\Gamma)$. Hence, the closure of $Q(\Gamma)$ in $U(I,\reals)$ is $Q(\tilde\Gamma)$.

We begin by examining the simplest orbit in $U(I,\reals)$. Let $q_0\equiv 1$ denote the constant function. What can we say about the orbit $[q_0]$? 
\medskip
\begin{lemma} $[q_0]=\hbox{Cl}(Q(\Gamma))=\{g\in U(I,\reals):  \hbox{ for all } t\in I, g(t)\geq 0 \hbox{ a.e.}\}=Q(\tilde\Gamma)$. \end{lemma}

\begin{proof} Given $\gamma\in\Gamma$, it's immediate that $(q_0*\gamma)(t)=\sqrt{\gamma'(t)}=Q(\gamma)(t)$. The lemma then follows immediately from the previous paragraph. 
\end{proof}

\begin{corollary}\label{warpinterval} Given any $g_1\in U(I,\reals)$ and $g_2\in U(I,\reals)$ satisfying $g_1(t)\geq0$ a.e. and $g_2(t)\geq0$ a.e., there is a sequence $\{\gamma_n\}$ in $\Gamma$ such that $g_1*\gamma_n\to g_2$ in the $L^2$ metric.
\end{corollary}

\begin{proof} Since $g_1$ and $g_2$ are both elements of Cl$(Q(\Gamma))=[q_0]$, it follows that $[g_1]=[g_2]=[q_0]$. The corollary follows immediately.
 \end{proof}

Let us think of $q_0\equiv 1$ as the ``north pole" of $U(I,\reals)$, and $-q_0\equiv -1$ as the ``south pole". Then the ``upper hemisphere" is $\{g\in U(I,\reals):\int_0^1\,g(t)\,dt\geq0\}$, the ``lower hemisphere" is $\{g\in U(I,\reals):\int_0^1\,g(t)\,dt\leq0\}$, and the ``equatorial sphere" is $\{g\in U(I,\reals):\int_0^1\,g(t)\,dt=0\}$. 

We next note that $[q_0]$ is completely contained in the upper hemisphere, and in fact does not intersect the equatorial sphere. That's because if $\int_0^1\,g(t)^2\,dt=1$ but $\int_0^1\,g(t)\,dt\leq0$, then it is clear that there must be a set of measure greater than zero on which $g(t)<0$. 

We now observe that even though $[q_0]$ does not intersect the equatorial sphere, it does contain points arbitrarily close to the equatorial sphere! To see this, let $0<\epsilon<1$, and define $g_\epsilon\in U(I,\reals)$ by
\begin{equation*}
g_\epsilon(t)=
\begin{cases}
1/\sqrt{\epsilon}& \text{for $t<\epsilon$}
\\
0&\text{for $t\geq\epsilon$}
\end{cases}
\end{equation*} 
and define $v_\epsilon\in U(I,\reals)$ by 
\begin{equation*}
v_\epsilon(t)=
\begin{cases}
\sqrt{\frac{1-\epsilon}{\epsilon}}& \text{for $t<\epsilon$}
\\
-\sqrt{\frac{\epsilon}{1-\epsilon}}& \text{for $t\geq\epsilon$}.
\end{cases}
\end{equation*}
Clearly, $g_\epsilon\in[q_0]$, $v_\epsilon$ is in the equatorial sphere, and $d(g_\epsilon,v_\epsilon)=\sqrt{2-2\sqrt{1-\epsilon}}$, which can be made as small as desired by taking $\epsilon$ small.

The next lemma will calculate the distance between two specific orbits, and will provide actual orbit representatives that realize this distance.

\begin{remark} Since $\Gamma$ and $\tilde\Gamma$ act by isometries on $L^2(I,\reals^N)$, it follows that minimizing the distance between orbit representatives of $[q]$ and $[w]$ is equivalent to maximizing their $L^2$ inner products.
\end{remark}

\begin{lemma}\label{geod1} Suppose the unit interval is expressed as a disjoint union of measurable sets, $I=A\cup B$, where  $A$ is assumed to have measure $a$. (For purposes of visualization, the reader may want to keep in mind the case in which $A$ and $B$ are simply subintervals of $I$.) Define
\begin{equation*}
w(t)=
\begin{cases}
1& \text{for $t\in A$}
\\
-1& \text{for $t\in B$.}
\end{cases}
\end{equation*} 
Then $d([q_0],[w])=\sqrt{2-2\sqrt{a}}$, and this distance is realized by the orbit representatives $w\in[w]$ and $q_A\in[q_0]$, where we define
\begin{equation*}
q_A(t)=
\begin{cases}
1/\sqrt{a}& \text{for $t\in A$}
\\
0& \text{for $t\in B$.}
\end{cases}
\end{equation*}
\end{lemma}

\begin{proof} First, it's easily verified that $\int_0^1w(t)q(t)\,dt=\sqrt{a}$ (implying $d(q,w)=\sqrt{2-2\sqrt{a}}$). So we just need to prove that this is the maximum over all representatives of $[q_0]$. In what follows, the key step will be the Cauchy Schwarz inequality, which states that for arbitrary $f$ and $g$ in $L^2(A)$, $|\int_A f(t)g(t)\,dt|\leq\left(\int_Af(t)^2\,dt\right)^{1/2}\left(\int_Ag(t)^2\,dt\right)^{1/2}.$

Continuing with the proof, choose an arbitrary $q\in[q_0]$. Then calculate 
$$\int_0^1q(t)w(t)\,dt
=\int_Aq(t)w(t)\,dt+\int_Bq(t)w(t)\,dt.$$
Clearly $\int_Bq(t)w(t)\,dt\leq 0$ since $q(t)\geq0$ and $w(t)\leq0$ on $B$. By Cauchy Schwarz, 
$\int_Aq(t)w(t)\,dt\leq1\cdot\sqrt{a}$. These two bounds imply that $\int_0^1q(t)w(t)\,dt\leq\sqrt{a}$. 
\end{proof}

We remark that since geodesics in $L^2(I,\reals^N)$ are straight lies, it is easy to write down a specific shortest geodesic from $q_A$ to $w$ in the lemma above. It's also easy to verify that the image of this geodesic in ${\mathcal S}(I,\reals)$ is a geodesic (in fact, a shortest geodesic) in this quotient space, according to the usual definition of geodesic used in metric spaces: 

{\bf Definition:} A {\em geodesic} in ${\mathcal S}(I,\reals)$ is a continuous function $\alpha:[0,L]\to{\mathcal S}(I,\reals)$ with the property that there exists a positive number $K$ such that for all $s\in[0,L]$, there exists $\epsilon>0$ such that for all $t_1, t_2\in(s-\epsilon,s+\epsilon)\cap[0,L]$, $d(\alpha(t_1),\alpha(t_2))=K|t_2-t_1|$.

Some fundamental problems regarding geodesics in ${\mathcal S}(I,\reals)$ are: (1) Given any two points in ${\mathcal S}(I,\reals)$ is there a geodesic joining them? Is there a shortest geodesic? Can we find this geodesic in some reasonable way, for some reasonable set of points in ${\mathcal S}(I,\reals)$? In Lemma \ref{geod1}, we found a precise description of the shortest geodesic between the particular orbit $[q_0]\in{\mathcal S}(I,\reals)$, and any orbit of the form $[w]$, where $w$ has the form
\begin{equation*}
w(t)=
\begin{cases}
1& \text{for $t\in A$}
\\
-1& \text{for $t\in B$.}
\end{cases}
\end{equation*}
However, in Theorem \ref{arclengthpar}, we proved that every element of $U(I,\reals)$ can be uniquely expressed as $w*\gamma$, where $w$ is of the form described above and $\gamma\in\tilde\Gamma$. (In fact, we proved a more general version of this for $U(I,\reals^N)$). Hence, we now know that for every orbit $[q]\in U(I,\reals)$, there is a unique shortest geodesic in ${\mathcal S}(I,\reals)$ joining $[q]$ to the particular orbit $[q_0]$, and in fact we have given a precise description of that geodesic.

\begin{lemma}\label{fixup} Let $0\leq a<b\leq 1$ and suppose $q$ and $w$ are two elements of  $U(I,\reals)$ with the following three properties:
\begin{enumerate}
\item For all $t\in(a,b)$, $q(t)\geq0$ and $w(t)\geq0$.
\item $\int_a^b q(t)^2dt=\int_a^b w(t)^2dt$
\item For all $t\not\in(a,b)$, $q(t)=w(t)$.

\end{enumerate}
Then $[q]=[w]$.
\end{lemma}

\begin{remark} We may replace condition (1) by the assumption that for all $t\in(a,b)$, $q(t)\leq0$ and $w(t)\leq0$, and the Lemma still holds, with the same proof.
\end{remark}

\begin{proof} By Corollary \ref{warpinterval}, there is a sequence $\{\lambda_n\}$ of absolutely continuous homeomorphisms $[a,b]\to[a,b]$ with the property that $\lambda_n^{-1}(0)$ has measure 0, such that in $L^2([a,b])$, $q*\lambda_n\to w$. (Note that in the proof of Corollary \ref{warpinterval} it makes no difference that we have changed the interval from $I$ to $[a,b]$ and changed the value of $\int_a^b q(t)^2dt=\int_a^b w(t)^2dt$ from 1 to whatever it is.) Then for each $n$, extend $\lambda_n$ to  $\gamma_n\in\Gamma$, by extending it as the identity outside $[a,b]$. It is then clear that in $L^2(I)$, $q*\gamma_n\to w$.  
\end{proof}

\begin{lemma}\label{charstep} Let $q\in L^2(I,\reals)$. Then $[q]\in {\mathcal S}(I,\reals)_{st}$ if and only if there is a finite sequence $0=t_0<t_1<\dots<t_n=1$ such that for each $j$, either $q(t)\geq0$ for all $t\in [t_{j-1},t_j]$ (a.e.), or $q(t)\leq0$ for all $t\in [t_{j-1},t_j]$ (a.e.).
\end{lemma}

\begin{proof} First, suppose there exists a finite sequence $0=t_0<t_1<\dots<t_n=1$ such that for each $j$, either $q(t)\geq0$ for all $t\in [t_{j-1},t_j]$, or $q(t)\leq0$ for all $t\in [t_{j-1},t_j]$. Then, by successive applications of Lemma \ref{fixup} we can find a new function in $[q]$ that is constant on each $[t_{j-1},t_j]$. It follows that $[q]\in{\mathcal S}(I,\reals)_{st}$. 

On the other hand, suppose $[q]\in{\mathcal S}(I,\reals)_{st}$. Let $w$ be a step function in $[q]$. It follows that there exists a sequence $\{\gamma_i\}$ in $\Gamma$ such that $w*\gamma_i\to q$ in $U(I,\reals)$ with respect to the $L^2$ metric. Because of this convergence in $L^2$, we may choose a subsequence of $\{\gamma_i\}$ such that $w*\gamma_i\to q$ a.e. Since $w$ is a step function, we can find a sequence $0=t_0<t_1<\dots<t_n=1$ such that $w$ is constant on each interval $(t_{j-1},t_j)$. For each  $j\in\{1,\dots,n-1\}$ and $i=1,2,3,\dots$, let $t_{j,i}=\gamma_i^{-1}(t_j)$. By compactness of $I$, we can replace the sequence $\{\gamma_i\}$ by a subsequence with the property that for each fixed $j_0$, the sequence $\{t_{j_0,i}\}$ converges to a number $\tilde t_{j_0}\in I$, as $i\to\infty$. Now, fix  $j\in\{1,\dots,n\}$. We will show that either $q(t)\geq0$ a.e. or $q(t)\leq0$ a.e. for $t\in(\tilde t_{j-1},\tilde t_j)$. WLG, assume that $w(t)\geq0$ on $(t_{j-1},t_j)$. Let $\tilde t\in(\tilde t_{j-1},\tilde t_j)$. By definition of the action of $\Gamma$, we know that $w*\gamma_i(t)\geq0$ for all $t\in(t_{j-1,i},t_{j,i})$. Choose an $\epsilon>0$ such that there exists an $N$ such that for all $i>N$, $(\tilde t-\epsilon,\tilde t+\epsilon)\subset(t_{j-1,i},t_{j,i})$. It follows that for all $i>N$, $w*\gamma_i(t)\geq 0$ on $(\tilde t-\epsilon,\tilde t+\epsilon)$. Since $w*\gamma_i\to q$ a.e., it follows that $q(t)\geq 0$ almost everywhere in $(\tilde t_{j-1}-\epsilon,\tilde t_j+\epsilon)$. Thus, we have shown that every  $\tilde t\in(\tilde t_{j-1},\tilde t_j)$ has a neighborhood on which $q(t)\geq0$ a.e. Since a countable number of these neighborhoods cover $(\tilde t_{j-1},\tilde t_j)$, it follows that $q(t)\geq0$ a.e. in $(\tilde t_{j-1},\tilde t_j)$. Thus, we have produced a finite sequence $\{\tilde t_j\}$ such that on each $(\tilde t_{j-1},\tilde t_j)$, either $q(t)\geq0$ a.e. or $q(t)\leq0$ a.e. 
\end{proof}

\begin{lemma}\label{domstep}
Let $q_1,q_2\in L^2(I,\reals)$, and assume that $[q_1],[q_2]\in{\mathcal S}(I,\reals)_{st}$. Then there exist step functions $w_1\in[q_1]$ and $w_2\in[q_2]$ such that $d(w_1,w_2)\leq d(q_1,q_2)$.
\end{lemma}
\begin{proof} It suffices to find step functions $w_1\in[q_1]$ and $w_2\in[q_2]$ such that $\langle w_1,w_2\rangle\geq\langle q_1,q_2\rangle$. By applying Lemma \ref{charstep} to $q_1$ and $q_2$ and then taking the union of our finite $t_i$-sequences, we obtain a single finite sequence $0=t_0<\dots<t_n=1$ such that for each $i=1,\dots,n$, both $q_1$ and $q_2$ have constant sign on $[t_{i-1},t_i]$. 
By ``constant sign", we mean that on this interval either $q_1(t)\geq0$ a.e. or $q_1(t)\leq0$ a.e., and either  $q_2(t)\geq0$ a.e. or $q_2(t)\leq0$ a.e. We now alter $q_1$ and $q_2$ on each of these subintervals in the following way.
\begin{enumerate}
\item If $q_1$ and $q_2$ have the same sign on $[t_{i-1},t_i]$, then on this interval simply replace $q_1$ by the constant function which has the same sign and square integral as $q_1$. Do the same for $q_2$.
\item If $q_1$ and $q_2$ have different signs on $[t_{i-1},t_i]$, then replace $q_1$ on this interval by a function that is constant on the first half of the interval, zero on the second half of the interval, and has the same sign and square integral on $[t_{i-1},t_i]$ as $q_1$ does. Replace $q_2$ on this interval by a function that is zero on the first half of the interval, constant on the second half, and has the same sign and square integral as $q_2$ has on $[t_{i-1},t_i]$.
\end{enumerate}

Call the resulting functions $w_1$ and $w_2$. By performing these replacements one subinterval at a time, we see by Lemma \ref{fixup} that $[w_1]=[q_1]$ and $[w_2]=[q_2]$. Furthermore, by the Cauchy Schwarz inequality, we see that on each subinterval $[t_{i-1},t_i]$ where we performed alteration (1), $\int_{t_{i-1}}^{t_i}w_1(t)w_2(t)dt\geq \int_{t_{i-1}}^{t_i}q_1(t)q_2(t)dt$. For each interval where we performed alteration (2), the same inequality holds, since $\int_{t_{i-1}}^{t_i}w_1(t)w_2(t)dt=0$, while $\int_{t_{i-1}}^{t_i}q_1(t)q_2(t)dt\leq0$. This completes the proof of the lemma. 
\end{proof}

\begin{lemma}\label{squareintlim} Suppose for each $n$, $f_n:I\to \reals$ is an $L^2$ function, and $f:I\to \reals$ is also $L^2$. Suppose that $f_n\to f$ in the $L^2$ norm. Assume that for each $n$, we are given an $a_n$ in $I$ such that $\int_0^{a_n} (f_n)^2=1$ for all $n$. Furthermore, suppose that $a_n\to a$ in $I$. Then $\int_0^a f^2=1$.
\end{lemma}
\begin{proof} By passing to a subsequence, we can assume that $f_n\to f$ a.e. in $I$. Next, recall a variant of the dominated convergence theorem: if $|k_n|\leq h_n$, $k_n\to k$ a.e., $h_n\to h$ a.e., and $\int_0^1h_n\to \int_0^1h<\infty$, then $\int_0^1k_n\to\int_0^1k$. Since $\left\vert\| f_n\|_2-\|f\|_2\right\vert\leq\|f_n-f\|_2$, it follows that $\int_0^1f_n^2\to\int_0^1f^2$. Taking $k_n=\left|f_n^2-f^2\right|$, $k=0$, $h_n=f_n^2+f^2$, and $h=2f^2$, it follows that $\int_0^1\left|f_n^2-f^2\right|\to 0.$

We now compute 
$$\left\vert\int_0^{a_n}f_n^2-\int_0^af^2\right\vert=\left\vert\int_0^{a_n}(f_n^2-f^2)-\int_{a_n}^af^2\right\vert\leq
\int_0^1|f_n^2-f^2|+\int_{a_n}^af^2.$$
The first of these terms was proved to approach $0$ at the end of the last paragraph; the second approaches $0$ by the absolute continuity of the integral, since $a_n\to a$.   
\end{proof}

\begin{lemma}\label{sameshape} Suppose $q$ and $w$ are elements of $L^2(I,\reals)$, and assume that $0=t_0<t_1<\dots<t_n=1$ is a sequence such that for each $i$ both of the following statements are true:
\begin{itemize}
\item either $q(t)\geq 0$ for $t\in [t_{i-1},t_i]$ a.e. or $q(t)\leq 0$ for $t\in [t_{i-1},t_i]$ a.e. and
\item $\int_{t_{i-1}}^{t_i}q(t)^2dt>0$.
\end{itemize}
Then $w\in[q]$ if and only if there exists a sequence $0=\tilde t_0<\tilde t_1<\dots<\tilde t_n=1$ such that for each $i$, both of the following statements are true:
\begin{itemize}
\item either $q(t)\geq 0$ for $t\in [t_{i-1},t_i]$ a.e. and $w(t)\geq 0$ for $t\in [\tilde t_{i-1},\tilde t_i]$ a.e.  or $q(t)\leq 0$ for $t\in [t_{i-1},t_i]$ a.e. and $w(t)\leq 0$ for $t\in [\tilde t_{i-1},\tilde t_i]$ a.e. and
\item $\int_{\tilde t_{i-1}}^{\tilde t_i}w(t)^2dt=\int_{t_{i-1}}^{t_i}q(t)^2dt$.
\end{itemize}
\end{lemma}

\begin{proof}  We omit the proof; it is basically the same as the proof of Lemma \ref{charstep}, but uses Lemma \ref{squareintlim} to keep track of the square integrals. 
\end{proof}

\begin{lemma}\label{zerospeedinorbits} Let $q\in L^2(I,\reals)$. Then $q\tilde\Gamma\subset[q]$.
\end{lemma}

\begin{proof} We will prove this first for step functions, and then extend by density to all of $S$, so we start by assuming that $q$ is a step function, and let $\tilde\gamma\in\tilde\Gamma$. Let $0=t_0<t_1<\dots<t_n=1$ be the finite set of points at which $q(t)$ changes values. For each $i$, choose $\tilde t_i\in I$ such that $\tilde\gamma(\tilde t_i)=t_i$. Letting $w(t)=q*\tilde\gamma(t)$, it follows from integration by substitution that the hypotheses of Lemma \ref{sameshape} are satisfied. Hence, $w\in[q]$, which completes the proof for $q$ a step function.

Now, let $q\in L^2(I,\reals)$ be arbitrary, and let $\tilde\gamma\in\tilde\Gamma$. Let $\epsilon>0$ be given. By density, choose a step function $v\in U(I,\reals)$ such that $d(q,v)<\epsilon/3$. By the previous paragraph, we know that $v*\tilde\gamma\in[v]$; this means we can choose $\gamma\in\Gamma$ such that $d(v*\gamma,v*\tilde\gamma)<\epsilon/3$. By the triangle inequality, and the fact the $\tilde\Gamma$ acts by isometries, we then conclude that
\begin{equation*}
d(q*\tilde\gamma,q*\gamma)\leq d(q*\tilde\gamma,v*\tilde\gamma)+d(v*\tilde\gamma,v*\gamma) +d(v*\gamma,q*\gamma)<\epsilon,
\end{equation*}
which completes the proof of the current lemma.  
\end{proof}

Define a function $w\in U(I,\reals)$ to be in {\it standard form} if it is of the form described in Lemma 
\ref{geod1}, i.e.,
\begin{equation*}
w(t)=
\begin{cases}
1& \text{for $t\in A$}
\\
-1& \text{for $t\in B$}
\end{cases}
\end{equation*} 
where $I=A\cup B$ is a partition of $I$ into two disjoint measurable sets. 

\begin{lemma}\label{differentfunc}
Suppose $q$ and $w$ are both in standard form, and $q\neq w$ in $L^2$ (i.e., the set $\{t\in I:q(t)\neq w(t)\}$ has measure greater than $0$). Then $w\not\in[q]$.
\end{lemma}

\begin{proof} First we make a simple calculation. Let $0<a<U$, and define $p:[0,U]\to \reals$ by 
$p(t)=\sqrt {ta}+\sqrt{(U-t)(U-a)}$. Then $p(t)$ has a unique maximum at $p(a)=U$; in fact $p'(t)>0$ for $t<a$ and $p'(t)<0$ for $t>a$. This is an easy Calc I exercise!

Since $q$ and $w$ are both in standard form, we have two partitions $I=A\cup B$ and $I=C\cup D$ such that 
\begin{equation*}
q(t)=
\begin{cases}
1& \text{for $t\in A$}
\\
-1& \text{for $t\in B$}
\end{cases}
\end{equation*} 
and
\begin{equation*}
w(t)=
\begin{cases}
1& \text{for $t\in C$}
\\
-1& \text{for $t\in D$}.
\end{cases}
\end{equation*}
Let $\mu$ denote Lebesgue measure. The remainder of the proof will consist of considering the two cases  $\mu(A)\neq\mu(C)$ and $\mu(A)=\mu(C)$.

{\bf Case 1:} Assume $\mu(A)\neq\mu(C)$. Suppose $\gamma\in\Gamma$, and let $\tilde A=\gamma^{-1}(A)$ and $\tilde B=\gamma^{-1}(B)$. Using integration by substitution, we see that $\int_{\tilde A}(q*\gamma(t))^2dt=\int_A (q(t))^2 dt =\mu(A)$ and $\int_{\tilde B}(q*\gamma(t))^2dt=\int_B (q(t))^2 dt =\mu(B)$. Now compute:
\begin{equation*}
\int_0^1 q*\gamma(t)w(t) dt=\int_{\tilde A\cap C}q*\gamma(t)w(t) dt +\int_{\tilde A\cap D}q*\gamma(t)w(t) dt +\int_{\tilde B\cap C}q*\gamma(t)w(t) dt + \int_{\tilde B\cap D}q*\gamma(t)w(t) dt
\end{equation*}
\begin{equation*}
\leq\int_{\tilde A\cap C}q*\gamma(t)w(t) dt+\int_{\tilde B\cap D}q*\gamma(t)w(t) dt
\end{equation*}
since on $\tilde A\cap D$ and $\tilde B\cap C$, $q*\gamma$ and $w$ have opposite signs, so these two terms make a negative contribution to the integral. However, by the Cauchy Scharz inequality, 
\begin{equation*}
\int_{\tilde A\cap C}q*\gamma(t)w(t)dt \leq \sqrt{\int_{\tilde A\cap C}(q*\gamma(t))^2 dt} \sqrt{ \int_{\tilde A\cap C}(w(t))^2 dt} 
\end{equation*}
\begin{equation*}
\leq \sqrt{\int_{\tilde A}(q*\gamma(t))^2 dt} \sqrt{ \int_{C}(w(t))^2 dt}=\sqrt{\mu(A)\mu(C)}.
\end{equation*}
Similarly,
\begin{equation*}
\int_{\tilde B\cap D}q*\gamma(t)w(t) dt\leq\sqrt{\mu(B)\mu(D)}.
\end{equation*}
Combining these with the last inequality gives
\begin{equation*}
\int_0^1 q*\gamma(t)w(t) dt\leq \sqrt{\mu(A)\mu(C)}+\sqrt{\mu(B)\mu(D)}
=\sqrt{\mu(A)\mu(C)}+\sqrt{(1-\mu(A))(1-\mu(C))}
\end{equation*}
Since we are assuming here that $\mu(A)\neq\mu(C)$, it follows from the calculation we made at the beginning of this proof that this upper bound is strictly less than 1. Also, this upper bound is independent of which element $\gamma\in\Gamma$ we chose. Since $\langle q*\gamma,w\rangle$ has an upper bound that is strictly less than 1 on the orbit $q\Gamma$, it follows that $d(q*\gamma, w)$ has a lower bound that is greater than zero on this orbit. This finishes the proof in Case 1; we have shown that if $\mu(A)\neq\mu(C)$, then $w\not\in[q]$. 

{\bf Case 2:} Assume $\mu(A)=\mu(C)=U$ (so $\mu(B)=\mu(D)=1-U$).

Let $f,g\in AC_0$ be the absolutely continuous functions satisfying $Q(f)=q$ and $Q(g)=w$. Since $Q$ is a bijection, $f\neq g$, so there exists $z\in I$ such that $f(z)\neq g(z)$. By definition of $Q$, $f(z)=\mu(A\cap[0,z])-\mu(B\cap[0,z])$ and $g(z)=\mu(C\cap[0,z])-\mu(D\cap[0,z])$. Since $z=\mu(A\cap[0,z])+\mu(B\cap[0,z])=\mu(C\cap[0,z])+\mu(D\cap[0,z])$, we may conclude that $\mu(A\cap[0,z])\neq\mu(C\cap[0,z])$; without loss of generality, let's assume $\mu(A\cap[0,z])<\mu(C\cap[0,z])$ and hence $\mu(B\cap[0,z])>\mu(D\cap[0,z])$.


Let $\gamma\in\Gamma$ be arbitrary. 



Compute:
\begin{equation*}
\langle q*\gamma,w\rangle=\int_0^1q*\gamma(t)w(t)dt=\int_0^zq*\gamma(t)w(t)dt+
\int_z^1q*\gamma(t)w(t)dt
\end{equation*}
\begin{equation*}
\leq\int_{[0,z]\cap(\tilde A\cap C)}q*\gamma(t)w(t)dt+\int_{[0,z]\cap(\tilde B\cap D)}q*\gamma(t)w(t)dt
\end{equation*}
\begin{equation*}
+\int_{[z,1]\cap(\tilde A\cap C)}q*\gamma(t)w(t)dt+\int_{[z,1]\cap(\tilde B\cap D)}q*\gamma(t)w(t)dt
\end{equation*}
where this inequality follows because on the parts of the interval we left out, the contribution of the integrand is negative. Continuing by the Cauchy Schwarz inequality:
\begin{equation*}
\leq\sqrt{\int_{[0,z]\cap(\tilde A\cap C)}(q*\gamma(t))^2dt}\sqrt{\int_{[0,z]\cap(\tilde A\cap C)}w(t)^2dt}
+\sqrt{\int_{[0,z]\cap(\tilde B\cap D)}(q*\gamma(t))^2dt}\sqrt{\int_{[0,z]\cap(\tilde B\cap D)}w(t)^2dt}
\end{equation*}
\begin{equation*}
+\sqrt{\int_{[z,1]\cap(\tilde A\cap C)}(q*\gamma(t))^2dt}\sqrt{\int_{[z,1]\cap(\tilde A\cap C)}w(t)^2dt}
+\sqrt{\int_{[z,1]\cap(\tilde B\cap D)}(q*\gamma(t))^2dt}\sqrt{\int_{[z,1]\cap(\tilde B\cap D)}w(t)^2dt}
\end{equation*}
\begin{equation*}
\leq\sqrt{\int_{[0,z]\cap\tilde A}(q*\gamma(t))^2dt}\sqrt{\int_{[0,z]\cap C}w(t)^2dt}
+\sqrt{\int_{[0,z]\cap\tilde B)}(q*\gamma(t))^2dt}\sqrt{\int_{[0,z]\cap D}w(t)^2dt}
\end{equation*}
\begin{equation*}
+\sqrt{\int_{[z,1]\cap\tilde A}(q*\gamma(t))^2dt}\sqrt{\int_{[z,1]\cap C}w(t)^2dt}
+\sqrt{\int_{[z,1]\cap\tilde B}(q*\gamma(t))^2dt}\sqrt{\int_{[z,1]\cap  D}w(t)^2dt}
\end{equation*}
\begin{equation*}
=\sqrt{\int_{[0,\gamma(z)]\cap A}q(t)^2dt}\sqrt{\int_{[0,z]\cap C}w(t)^2dt}
+\sqrt{\int_{[0,\gamma(z)]\cap B)}q(t)^2dt}\sqrt{\int_{[0,z]\cap D}w(t)^2dt}
\end{equation*}
\begin{equation*}
+\sqrt{\int_{[\gamma(z),1]\cap A}q(t)^2dt}\sqrt{\int_{[z,1]\cap C}w(t)^2dt}
+\sqrt{\int_{[\gamma(z),1]\cap B}q(t)^2dt}\sqrt{\int_{[z,1]\cap  D}w(t)^2dt}
\end{equation*}
\begin{equation*}
=\sqrt{\mu([0,\gamma(z)]\cap A)\mu([0,z]\cap C)}+\sqrt{\mu([0,\gamma(z)]\cap B)\mu([0,z]\cap D)}
\end{equation*}
\begin{equation*}
+\sqrt{\mu([\gamma(z),1]\cap A)\mu([z,1]\cap C)}+\sqrt{\mu([\gamma(z),1]\cap B)\mu([z,1]\cap D)}
\end{equation*}
\begin{equation*}
=\sqrt{\mu([0,\gamma(z)]\cap A)\mu([0,z]\cap C)}+\sqrt{\mu([0,\gamma(z)]\cap B)\mu([0,z]\cap D)}
\end{equation*}
\begin{equation*}
+\sqrt{(U-\mu([0,\gamma(z)]\cap A))(U-\mu([0,z]\cap C))}+\sqrt{(1-U-\mu([0,\gamma(z)]\cap B))(1-U-\mu([0,z]\cap D))}.
\end{equation*}
Reversing the order of the middle two terms gives
\begin{equation}\label{bound}
=\sqrt{\mu([0,\gamma(z)]\cap A)\mu([0,z]\cap C)}+\sqrt{(U-\mu([0,\gamma(z)]\cap A))(U-\mu([0,z]\cap C))}
\end{equation}
\begin{equation*}
+\sqrt{\mu([0,\gamma(z)]\cap B)\mu([0,z]\cap D)}
+\sqrt{(1-U-\mu([0,\gamma(z)]\cap B))(1-U-\mu([0,z]\cap D))}.
\end{equation*}
Now there are two possibilities to consider: Either $\gamma(z)\leq z$ or $\gamma(z)\geq z$.

{\bf Case (i):} Assume $\gamma(z)\leq z$. In this case $\mu([0,\gamma(z)]\cap A)\leq\mu([0,z]\cap A)<\mu([0,z]\cap C)$. 

By the observations about $p(t)$ and $p'(t)$ made in the first paragraph of this proof, we conclude that the sum of the first two summands of expression \ref{bound} is bounded above by
\begin{equation*}
\sqrt{\mu([0,z]\cap A)\mu([0,z]\cap C)}+\sqrt{(U-\mu([0,z]\cap A))(U-\mu([0,z]\cap C))}
\end{equation*}
and that this bound is strictly less than $U$. Also, by the first paragraph of this proof, the sum of the third and fourth summands of expression \ref{bound} is bounded above by $1-U$. As a result, in Case (i), we have an upper bound for $\langle q*\gamma,w\rangle$ which is strictly less than 1, and is independent of $\gamma\in\Gamma$ (except for the condition that $\gamma(z)\leq z$). 

{\bf Case (ii):} Assume $\gamma(z)\geq z$. In that case, $\mu([0,\gamma(z)]\cap B)\geq\mu([0,z]\cap B)>\mu([0,z]\cap D)$.  Then, by the observations about $p(t)$ and $p'(t)$ made in the first paragraph of this proof, we conclude that the sum of the third and fourth summands of expression \ref{bound} is bounded above by 
\begin{equation*}
\sqrt{\mu([0,z]\cap B)\mu([0,z]\cap D)}
+\sqrt{(1-U-\mu([0,z]\cap B))(1-U-\mu([0,z]\cap D))}.
\end{equation*}
and that this bound is strictly less than $1-U$. Also, by the first paragraph of this proof, the sum of the first two summands of expression \ref{bound} is bounded above by $U$. Hence, in Case (ii), we also have an upper bound for $\langle q*\gamma,w\rangle$ which is strictly less than 1, and is independent of $\gamma\in\Gamma$ (except for the condition that $\gamma(z)\geq z$).

Taking the greater of these two upper bounds, we have an upper bound for $\langle q*\gamma,w\rangle$ that is strictly less than 1 and is completely independent of the choice of $\gamma\in\Gamma$. As a result, we have a lower bound for $d(q*\gamma,w)$ that is strictly greater than 0 and is completely independent of $\gamma\in\Gamma$. This proves that $w\not\in[q]$, and completes the proof of Lemma \ref{differentfunc}. \end{proof}

Let ${\mathcal SF}(I,\reals)=\{q\in U(I,\reals): q \hbox{ is in standard form.}\}$ By the unique arclength parametrization theorem quoted earlier (Theorem \ref{arclengthpar}), we can express $U(I,\reals)$ as a disjoint union as follows:
$$U(I,\reals)=\coprod_{w\in{\mathcal SF}}w\tilde\Gamma$$
By Lemma \ref{zerospeedinorbits}, $q\tilde\Gamma\in[q]$ for each $q\in{\mathcal SF}$. By Lemma \ref{differentfunc}, if $q,w\in{\mathcal SF}$ and $q\neq w$, then $q\not\in[w]$. It follows that $[q]\cap[w]=\emptyset$ and therefore $q\tilde\Gamma \cap[w]=\emptyset$. By the disjoint union above, since $[w]\cap q\tilde\Gamma=\emptyset$ for all $q\in{\mathcal SF}$ where $q\neq w$, it follows that $[w]\subset w\tilde\Gamma$. Combined with Lemma \ref{zerospeedinorbits}, this proves the following theorem.
\begin{thm}\label{orbitstructure}
For all $w\in{\mathcal SF}$, $[w]=w\tilde\Gamma$.
\end{thm}
\begin{corollary}\label{generalorbitstructure}
If $q\in L^2(I,\reals)$, and $q^{-1}(0)$ has measure $0$, then $[q]=q\tilde\Gamma$.
\end{corollary}

\begin{proof} First, assume that $q\in U(I,\reals)$. By Theorem \ref{arclengthpar}, we can write $q=w*\gamma=(w\circ\gamma)\sqrt{\gamma'}$, where $\gamma\in\tilde\Gamma$ and $|w(t)|=1$ for almost all $t\in I$. Since $q^{-1}(0)$ has measure $0$, it follows that $\{t\in I:\gamma'(t)=0\}$ has measure $0$, so $\gamma\in\Gamma$. Hence, we can write $w=q*\gamma^{-1}$ and, therefore, $[q]=[w]=w\tilde\Gamma=q\gamma^{-1}\tilde\Gamma=q\tilde\Gamma$, since $\gamma^{-1}\tilde\Gamma=\tilde\Gamma$. For $q\not\in U(I,\reals)$, the result also follows, since multiplication by a constant nonzero scalar is a homeomorphism that commutes with the action of $\tilde\Gamma$.

\end{proof}

\section{Orbit Structure in ${\mathcal S}(I,\reals^N)$}

In this section, we will extend Theorem \ref{orbitstructure} and Corollary \ref{generalorbitstructure} from $L^2(I,\reals)$ to $L^2(I,\reals^N)$. This has a few more technical difficulties than one might expect; hence it gets its own section. Just as in the case of $N=1$, we define $w\in U(I,\reals^N)$ to be in {\it standard form} if $|w(t)|=1$ for $t\in I$ a.e. Let ${\mathcal SF}(I,\reals^N)$ be the subset of $U(I,\reals^N)$ consisting of functions in standard form. By Theorem \ref{arclengthpar}, we can express $U(I,\reals^N)$ as the disjoint union
\begin{equation}
U(I,\reals^N)=\coprod_{w\in{\mathcal SF}(I,\reals^N)}w\tilde\Gamma.
\end{equation}

We define $q\in L^2(I,\reals^N)$ to be a {\it step function} if we can express $I$ as a finite union of disjoint sub-intervals in such a way that $q$ is constant on each of these subintervals. Again, we don't care what happens at the endpoints of the subintervals since $q$ is only well-defined almost everywhere. 

\begin{lemma}\label{orbinclude}
For all $q\in L^2(I,\reals^N)$, $q\tilde\Gamma\subset[q]$.
\end{lemma}
 
\begin{proof} The proof follows the same lines as Lemma \ref{zerospeedinorbits}, with some minor adjustments. We first prove the lemma for step functions, then extend by density to all of $U(I,\reals^N)$. 
 
 Let $q\in U(I,\reals^N)$ be a step function, and let $\gamma\in\tilde\Gamma$. Choose $0=t_0<t_1<\dots<t_n=1$ such that $q$ is constant on each interval $(t_k,t_{k+1})$. Choose $0=s_0<s_1<\dots<s_n=1$ so that $\gamma(s_k)=t_k$ for each $k=0,1,\dots,n$. For each $k$, define $\gamma_k=\gamma|_{[s_{k-1},s_k]}$. Clearly, for each $k$, the function $\sqrt{\gamma_k'}$ is an element of the sphere of radius $\sqrt{t_k-t_{k-1}}$ centered at 0 in $L^2[s_{k-1},s_k]$. Another element of this sphere is the constant function $w_k(t)=\sqrt{(t_k-t_{k-1})/(s_k-s_{k-1})}$. For $j=1,2,3,\dots$, let $\{w_{k,j}\}$ be any sequence of functions along the geodesic arc from $w_k$ to $\sqrt{\gamma_k'}$ in this sphere such that $\lim_{j\to\infty}w_{k,j}=\sqrt{\gamma_k'}$ in $L^2[s_{k-1},s_k]$. Note that for each $j$, $w_{k,j}(s)>0$ for all $s\in[s_{k-1},s_k]$. Finally, define $\gamma_{k,j}(s)=t_{k-1}+\int_{s_{k-1}}^s(w_{k,j}(u))^2\,du$ for all $s\in[s_{k-1},s_k]$.
Clearly, for all $k$ and $j$, $\gamma_{k,j}$ is an absolutely continuous, monotone homeomorphism from $[s_{k-1},s_k]$ to $[t_{k-1},t_k]$. For each $j$, define $\gamma_j:I\to I$ by setting $\gamma_j(s)=\gamma_{k,j}(s)$ for all $s\in[s_{k-1},s_k]$. Clearly, $\gamma_j\in\Gamma$ for all $j$. Assuming that $q(t)=c_k$ on $[t_{k-1},t_k]$, it follows that on $[s_{k-1},s_k]$, $q*\gamma_j(s)=\sqrt{\gamma_{k,j}'(s)}c_k=w_{k,j}(s)c_k$. By definition of $w_{k,j}$, it follows that in $L^2[s_{k-1},s_k]$, $\lim_{j\to\infty}q*\gamma_{k,j}=\lim_{j\to\infty}w_{k,j}c_k=\sqrt{\gamma_k'}c_k=q*\gamma$. Since this limit holds in each subinterval separately, it must hold in $L^2(I)$. Thus, we have produced a sequence $\gamma_j$ in $\Gamma$ such that $q*\gamma_j\to q*\gamma$, which proves the lemma for step functions. The extension from step functions to general elements of $L^2(I,\reals^N)$ is exactly the same as in the proof of Lemma \ref{zerospeedinorbits}, so we omit it. 
\end{proof}

Next, we want to prove that for all functions $w\in U(I,\reals^N)$ that are in standard form, $[w]\in w\tilde\Gamma$. Such a function $w$ can vanish only on a set of measure zero; however, each component function of $w$ may vanish on a set of measure greater than zero. In order to make our proof run more smoothly, it is helpful to prove that given such a $w$, we can rotate it (using a matrix in $O(n,\reals)$) to obtain a new function in $U(I,\reals^N)$ with the property that all of its component functions vanish only on a set of measure zero. The next few lemmas prove that such a rotation exists. 

\begin{lemma}\label{orthohyper} Let $q:I\to \reals^N$ (where $n\geq 1$) be an $L^2$ function with the property that $q^{-1}(0)$ has measure $0$. Then there exist $N$ pairwise orthogonal $(N-1)$-dimensional linear subspaces of $\reals^N$, which we denote by $H_1,\dots,H_N$, with the property that $q^{-1}(H_k)$ has measure zero for $k=1,\dots,N$. 
\end{lemma}

\begin{proof} First, we note that there exist at most countably many lines $l$ through the origin in $\reals^N$ with the property that $q^{-1}(l)$ has nonzero measure. To see this, note that as $l$ varies over all lines through the origin of $\reals^N$, the sets $q^{-1}(l-\{0\})$ are disjoint from each other. But it follows that for all but a countable set of these lines, the set $q^{-1}(l-\{0\})$ has measure 0. This is because $I$ cannot contain an uncountable collection of pairwise disjoint subsets, all with measure greater than 0. (The proof of this is easy: Let ${\mathcal C}$ denote a collection of pairwise disjoint measurable subsets of $I$. For each integer $n>0$, define the subcollection ${\mathcal C}_n\subset{\mathcal C}$ by ${\mathcal C}_n=\{C\in{\mathcal C}:\mu(C)>
\frac{1}{n}\}$. Clearly, the cardinality of ${\mathcal C}_n$ is at most $n$. The subcollection of ${\mathcal C}$ consisting of all sets of measure greater than 0 is just the union of these sets ${\mathcal C}_n$, which is countable.)  It follows that for all but a countable set of lines $l$ through the origin in $\reals^N$, $\mu(q^{-1}(l))=\mu(q^{-1}(l-\{0\}))=0$. Note that this last equation used the fact that $\mu(q^{-1}(0))=0$.

We now construct a sequence $0\subset P_1\subset P_2\subset\dots\subset P_{N-1}$ of linear subspaces of $\reals^N$, such that for each $k$, $P_k$ has dimension $k$ and $\mu(q^{-1}(P_k))=0$. We construct this sequence inductively. In the first paragraph of this proof, we showed that there exists a line $l$ through the origin such that $q^{-1}(l)$ has measure $0$. Let $P_1$ be any such $l$. For the inductive step, assume we have already constructed $0\subset P_1\subset\dots\subset P_k$ satisfying the conditions, where $k<N-1$. Choose an orthonormal basis $\{u_1,\dots,u_k\}$ of  $P_{k}$. Let $u$ and $v$ be any orthogonal pair of unit vectors, both in the orthogonal complement of $P_k$ in $\reals^N$. For each real number $\theta\in[0,\pi)$, let 
$S_{\theta}$ to be the linear span of $\{u_1,\dots,u_{k},(\cos\theta)u+(\sin\theta)v\}$. Clearly the subsets of $\reals^N$ in the collection $\{S_\theta-P_k\}_{\theta\in[0,\pi)}$ are pairwise disjoint, and hence the subsets of $I$ in the collection $\{q^{-1}(S_\theta-P_k)\}_{\theta\in[0,\pi)}$ are also pairwise disjoint. It follows that for all but a countable set of $\theta\in[0,\pi)$, $\mu(q^{-1}(S_\theta-P_k))=0$.
But since $\mu(q^{-1}(P_k))=0$, we know that $\mu(q^{-1}(S_\theta))=\mu(q^{-1}(S_\theta-P_k))=0$ for all but a countable set of $\theta$. Choosing one of these $\theta$, we then set $P_{k+1}=S_\theta$, completing the inductive step of the construction of our sequence $0\subset P_1\subset P_2\subset\dots\subset P_{N-1}$. 

We now prove Lemma \ref{orthohyper} by induction. It is trivially true for $N=1$. Assume the lemma has been proved for functions $I\to \reals^{N-1}$, and now suppose we are given an $L^2$ function $q:I\to \reals^N$ such that $q^{-1}(0)$ has measure $0$. Using the last paragraph, construct a sequence $0\subset P_1\subset P_2\subset\dots\subset P_{N-1}$ of linear subspaces of $\reals^N$, such that for each $k$, $P_k$ has dimension $k$ and $\mu(q^{-1}(P_k))=0$. We are now going to make an adjustment to $P_{N-1}$. Let $\{u,v\}$ be an orthonormal basis of the orthogonal complement of $P_{N-2}$ in $\reals^N$. For each $\theta\in[0,\pi)$, let $B_\theta$ be the hyperplane spanned by $P_{N-2}$ and the vector $(\cos\theta)u+(\sin\theta)v$, and let $l_\theta$ be the orthogonal complement of $B_\theta$ in $\reals^N$. Since we know that $q^{-1}(P_{N-2})$ has measure zero, we can argue just as in the last paragraph to show that for all but a countable set of $\theta\in[0,\pi)$, $q^{-1}(B_\theta)$ has measure zero. Similarly, for all but a countable set of $\theta$, $q^{-1}(l_\theta)$ has measure 0. Since the union of two countable sets is countable, it follows that we can choose a $\theta_0\in[0,\pi)$ such that both $q^{-1}(B_{\theta_0})$ and $q^{-1}(l_{\theta_0})$ have measure 0. Set $H_N=B_{\theta_0}$ and set $l=l_{\theta_0}$. Define $\tilde q:I\to H_N$ by $\tilde q=\Pi\circ q$, where $\Pi$ denotes orthogonal projection $\reals^N\to H_N$. Clearly, $\tilde q^{-1}(0)=q^{-1}(l)$ has measure $0$. Also, since projection decreases norms,  $\tilde q$ is still $L^2$. By the induction hypothesis, we can find a pairwise orthogonal set of $N-1$ subspaces of $H_N$, which we will denote by $\tilde H_1,\dots,\tilde H_{N-1}$, each of dimension $n-2$, such that $\tilde q^{-1}(\tilde H_k)$ has measure 0 for $k=1,\dots,N-1$. Now define $H_k=\tilde H_k\oplus l$ for each $k=1,\dots,N-1$. Clearly, $H_1,\dots,H_N$ comprise a set of $N$ pairwise orthogonal $(N-1)$-dimensional subspaces of $\reals^N$. Also, for $k=1,\dots,N-1$, $q^{-1}(H_k)=\tilde q^{-1}(\tilde H_k)$ has measure 0. Since we already know that $q^{-1}(H_N)$ has measure 0, this completes the proof of Lemma \ref{orthohyper}. 
\end{proof}
Let us think of the elements of $\reals^N$ as column vectors, so if $q\in L^2(I,\reals^N)$, we can write $q(t)=(q_1(t),\dots,q_N(t))'$, where the ``prime" denotes the matrix transpose. Then the group of orthogonal matrices $O(N,\reals)$ acts on $\reals^N$, $AC_0(I,\reals^N)$, $L^2(I,\reals^N)$ and $U(I,\reals^N)$ from the left in the obvious way. Furthermore, this action on $L^2(I,\reals^N)$ preserves the $L^2$ inner product, and hence takes $U(I,\reals^N)$ to itself. It also commutes with the bijection $Q:AC_0(I,\reals^N)\to L^2(I,\reals^N)$ defined earlier in this paper. In addition, the left actions of $O(N,\reals)$ on $AC_0(I,\reals^N)$, $L^2(I,\reals^N)$ and $U(I,\reals^N)$ commute with the right actions of $\Gamma$ and $\tilde\Gamma$ on all three of these spaces. 

\begin{lemma}\label{goodorthochange} Suppose $q\in L^2(I,\reals^N)$, where we write $q(t)=(q_1(t),\dots,q_N(t))'$, and assume that $q^{-1}(0)$ has measure zero. Then there exists a matrix $A\in O(N,\reals)$, such that if we define $\tilde q(t)=(\tilde q_1(t),\dots,\tilde q_N(t))'$ by $\tilde q(t)=Aq(t)$, then $\tilde q_k^{-1}(0)$ has measure $0$ for all $k=1,\dots, N$.
\end{lemma}
\begin{proof} By Lemma \ref{orthohyper}, we can choose $N$ pairwise orthogonal $(N-1)$-dimensional linear subspaces of $\reals^N$, which we denote by $H_1,\dots,H_N$, with the property that $q^{-1}(H_k)$ has measure zero for $k=1,\dots,N$. For each $k$, let $u_k\in \reals^N$ be a unit vector orthogonal to $H_k$. Clearly, $\{u_1,\dots,u_N\}$ forms an orthonormal basis for $\reals^N$. Let $\{e_1,\dots,e_N\}$ denote the standard basis for $\reals^N$. Let $A\in O(N,\reals)$ be the unique matrix satisfying $Au_k=e_k$ for all $k$. Define $\tilde q(t)=(\tilde q_1(t),\dots,\tilde q_N(t))'$ by $\tilde q(t)=Aq(t)$. By definition of $A$, it is immediate that $\tilde q_k^{-1}(0)=\tilde q^{-1}(e_k^\perp)=q^{-1}(u_k^\perp)=q^{-1}(H_k)$ has measure zero.
\end{proof}

\begin{lemma}\label{paramapproach} Suppose $q\in L^2(I,\reals)$ satisfies $\mu(q^{-1}(0))=0$. Also, assume that $\{\tau_k\}$ is a sequence in $\tilde\Gamma$, and $\tau\in\tilde\Gamma$. If $\lim_{k\to\infty}q*\tau_k=q*\tau$ (with respect to the $L^2$ metric), then for all $t\in I$, $\lim_{k\to\infty}\tau_k(t)=\tau(t)$.
\end{lemma}

\begin{proof} First, we observe that it suffices to prove the lemma under the additional assumption that $|q(t)|=1$ for almost all $t\in I$. For, suppose we have already completed the proof under this extra assumption. By Theorem \ref{arclengthpar}, we know that we can write $q=w*\gamma$ where $|w(t)|$ is constant for almost all $t\in I$ and $\gamma\in\tilde\Gamma$. Since we are assuming that $\mu(q^{-1}(0))=0$, it follows that $\gamma\in\Gamma$, since $\dot\gamma(t)=0$ only for $t$ in a set of measure 0. Now, since we are assuming that $\lim_{k\to\infty}q*\tau_k=q*\tau$, it follows that $\lim_{k\to\infty}w*(\gamma\tau_k)=w*(\gamma\tau)$ in $L^2$. By the version of the lemma that we are assuming to be proved, it follows that for all $t\in I$, $\lim_{k\to\infty}\gamma(\tau_k(t))=\gamma(\tau(t))$. Since $\gamma\in\Gamma$ is continuous and bijective, so is $\gamma^{-1}$, hence we conclude that for all $t\in I$, $\lim_{k\to\infty}\tau_k(t)=\tau(t)$. 

So we now prove the lemma with the assumption that $|q(t)|=1$ for almost all $t\in I$. Assume that $\lim_{k\to\infty}q*\tau_k=q*\tau$ with respect to the $L^2$ metric. We proceed by contrapositive. Suppose there exists a $t_0$ for which $\lim_{k\to\infty}\tau_k(t_0)\neq\tau(t_0)$. (This includes the possibility that $\lim_{k\to\infty}\tau_k(t_0)$ does not exist.) Therefore, there exists $\epsilon>0$ such that for all $M>0$, there exists a $k>M$ with $|\tau_k(t_0)-\tau(t_0)|>\epsilon$. Consider a $\tau_k$ such that $|\tau_k(t_0)-\tau(t_0)|>\epsilon$. There are two cases.

{\bf Case 1:} Suppose $\tau_k(t_0)<\tau(t_0)-\epsilon$. Compute $\langle q*\tau_k,q*\tau\rangle=$
$$\int_0^1(q*\tau_k)(t)(q*\tau)(t)dt=\int_0^{t_0}(q*\tau_k)(t)(q*\tau)(t)dt+\int_{t_0}^1(q*\tau_k)(t)(q*\tau)(t)dt$$
$$\leq\sqrt{\int_0^{t_0}(q*\tau_k)(t)^2dt}\sqrt{\int_0^{t_0}(q*\tau)(t)^2dt}+
\sqrt{\int_{t_0}^{1}(q*\tau_k)(t)^2dt}\sqrt{\int_{t_0}^{1}(q*\tau)(t)^2dt}$$
(by Cauchy Schwarz)
$$=\sqrt{\int_0^{\tau_k(t_0)}q(t)^2dt}\sqrt{\int_0^{\tau(t_0)}q(t)^2dt}+
\sqrt{\int_{\tau_k(t_0)}^{1}q(t)^2dt}\sqrt{\int_{\tau(t_0)}^{1}q(t)^2dt}$$
(integrating by substitution)
$$=\sqrt{\tau_k(t_0)}\sqrt{\tau(t_0)}+\sqrt{1-\tau_k(t_0)}\sqrt{1-\tau(t_0)}$$
(using $|q(t)|=1$ for almost all $t\in I$).

By the calculation at the beginning of the proof of Lemma \ref{differentfunc}, this last quantity is less than
$$\sqrt{\tau(t_0)-\epsilon}\sqrt{\tau(t_0)}+\sqrt{1-(\tau(t_0)-\epsilon)}\sqrt{1-\tau(t_0)}$$
which, in turn, is less than 1. Define $M_\epsilon=\sqrt{\tau(t_0)-\epsilon}\sqrt{\tau(t_0)}+\sqrt{1-(\tau(t_0)-\epsilon)}\sqrt{1-\tau(t_0)}$. So in Case 1, we have shown that $\langle q*\tau_k,q*\tau\rangle<M_\epsilon<1$.

{\bf Case 2:} Suppose $\tau_k(t_0)>\tau(t_0)+\epsilon$. We skip the very similar details; the end result is that we produce an $m_\epsilon$ such that $\langle q*\tau_k,q*\tau\rangle<m_\epsilon<1$. Letting $\theta_\epsilon=\min\{\sqrt{2-2M_\epsilon},\sqrt{2-2m_\epsilon}\}>0$, we see that $d(q*\tau_k,q*\tau)>\theta_\epsilon>0$. It follows that $q*\tau_k\not\to q*\tau$ in the $L^2$ metric, which proves the contrapositive and completes the proof of the lemma. 
\end{proof}

\begin{thm} Let $q:I\to \reals^N$ be any $L^2$ function such that $|q(t)|=1$ for almost all $t\in I$. Then $[q]=q\tilde\Gamma$. 
\end{thm}
\begin{proof} By Lemma \ref{orbinclude}, we only need to prove that $[q]\subset q\tilde\Gamma$.

Since $O(N,\reals)$ acts on $L^2(I,\reals^N)$ by isometries and this action commutes with the action of $\tilde\Gamma$, we may assume by Lemma \ref{goodorthochange} that for each $i=1,\dots,N$, $q_i^{-1}(0)$ has measure 0. Suppose $v\in[q]$; then there exists a sequence $\{\gamma_k\}$ in $\Gamma$ such that $\lim_{k\to\infty} q*\gamma_k=v$ (with respect to the $L^2$ metric). It follows that for each $i=1,\dots,N$, $\lim_{k\to\infty} q_i*\gamma_k=v_i$. For each $i$, we can write $q_i=w_i*\sigma_i$, where $\sigma_i\in\tilde\Gamma$ and $w_i(t)=\pm 1$ for almost all $t$, by Theorem \ref{arclengthpar}. Then for each $i$, $\lim_{k\to\infty} w_i*(\sigma_i\gamma_k)=v_i$. Since $w_i\tilde\Gamma=[w_i]$ is a closed set, it follows that there exists $\tilde\gamma_i\in\tilde\Gamma$ such that $v_i=w_i*\tilde\gamma_i$. Now, since $q_i(t)=(w_i*\sigma_i)(t)$ vanishes only on a set of measure 0, it follows that $\dot\sigma_i(t)$ also vanishes on a set of measure 0. Hence, $\sigma_i\in\Gamma$. Hence, we may write $w_i=q_i*\sigma_i^{-1}$, where $\sigma_i^{-1}\in\Gamma$. Therefore, $v_i=q_i*(\sigma_i^{-1}\tilde\gamma_i)$, where $\sigma_i^{-1}\tilde\gamma_i\in\tilde\Gamma$. Letting $\tau_i=\sigma_i^{-1}\tilde\gamma_i$, we have now proven that if $v\in[q]$, then for each $i$, we can write $v_i=q_i*\tau_i$, where each $\tau_i\in\tilde\Gamma$. All that is left to prove is that $\tau_1=\dots=\tau_n$; it will then follow that $v=q*\tau$ where $\tau=\tau_1=\dots=\tau_N$. However Lemma \ref{paramapproach} implies immediately that for each $t\in I$, $\tau_i(t)=\lim_{k\to\infty}\gamma_k(t)$, which implies that $\tau=\tau_1=\dots=\tau_n$. This completes the proof of the theorem. \end{proof}

\begin{corollary}\label{generalRnorbitstructure}
If $q\in L^2(I,\reals^N)$, and $q^{-1}(0)$ has measure $0$, then $[q]=q\tilde\Gamma$.
\end{corollary}

The proof is the same as that of Corollary \ref{generalorbitstructure}, so we omit it.

\begin{corollary}\label{maxdist}
If $q_1,q_2\in U(I,\reals^N)$, then $d([q_1],[q_2])\leq\sqrt{2}$. 
\end{corollary}
\begin{proof} Define two elements $\gamma_1,\gamma_2\in\tilde\Gamma$ as follows:
\begin{equation*}
\gamma_1(t)=
\begin{cases}
0& \text{for $0\leq t\leq .5$}
\\
2t-1& \text{for $.5<t\leq 1$}
\end{cases}
\end{equation*}
\begin{equation*}
\gamma_2(t)=
\begin{cases}
2t& \text{for $0\leq t\leq .5$}
\\
1& \text{for $.5<t\leq 1$}
\end{cases}
\end{equation*}
Then an easy computation shows that $\langle q_1*\gamma_1,q_2*\gamma_2\rangle=0$, so $d(q_1*\gamma_1,q_2*\gamma_2)=\sqrt{2}$. The corollary follows.
\end{proof}

For $U(I,\reals)$, it's interesting to note that this maximum distance is actually achieved by $[q_1]$ and $[q_2]$, where $q_1(t)\equiv1$ and $q_2(t)\equiv-1$. It is also true that these are the {\it only} two points in the image of $U(I,\reals)$ in ${\mathcal S}(I,\reals)$ that achieve this maximum distance! This follows from the material in Sections 7 and 8 for step functions, and then for arbitrary functions by the fact that the step functions are dense.

\section{Optimal Matching for Step Functions $I\to \reals^N$}

Given two elements $[w_1]$ and $[w_2]$ of ${\mathcal S}(I,\reals^N)$, it is a basic problem to calculate the distance between them and to find a minimal geodesic joining them, if such a geodesic exists. The most straightforward way to do this is to find elements $q_1\in[w_1]$ and $q_2\in[w_2]$ such that $d(q_1,q_2)=d([w_1],[w_2])$. In this section, we will prove that if $[w_1]$ and $[w_2]$ are elements of ${\mathcal S}(I,\reals^N)$, and at least one of them is an element of ${\mathcal S}_{st}(I,\reals^N)$, then there exist $q_1\in[w_1]$ and $q_2\in[w_2]$ such that $d(q_1,q_2)=d([w_1],[w_2])$. We will also prove that if both $[w_1]$ and $[w_2]$ are elements of ${\mathcal S}_{st}(I,\reals^N)$, then these representatives $q_1$ and $q_2$ can both be taken to be step functions. We begin with a lemma.

\begin{lemma}\label{BasicConstantMatch}
Let $q\in L^2(I,\reals^N)$, and let $w:I\to\reals^N$ be a constant map, $w(t)=w_0$. Express $I$ as a disjoint union of two measurable sets, $I=A\cup B$, where $A=\{t\in I:q(t)\cdot w_0\geq 0\}$ and $B=\{t\in I:q(t)\cdot w_0< 0\}$.

Then
\begin{equation*}
\sup_{\tilde q\in[q],\tilde w\in[w]}\langle \tilde q, \tilde w\rangle=
\sqrt{\int_A(q(t)\cdot w_0)^2dt}.
\end{equation*}
If $q(t)\cdot w_0=0$ almost everywhere on $I$, then this supremum is $0$, and is realized by any $\tilde q\in[q]$ and $\tilde w\in[w]$. If it is not true that $q(t)\cdot w_0=0$ almost everywhere on $I$, then this supremum is realized by $\tilde q=q$ and $\tilde w=w*\gamma$, where $\gamma(t)=\int_0^tF(u)du$, and $F(u)$ is defined by  

\begin{equation*}
F(u)=
\begin{cases}
0& \text{for $u\in B$}
\\
\frac{(q(u)\cdot w_0)^2}{\int_A(q(u)\cdot w_0)^2du}& \text{for $u\in A$}
\end{cases}
\end{equation*}

\end{lemma}

\begin{proof} Since $q\Gamma$ is dense in $[q]$ and $w\Gamma$ is dense in $[w]$ and the group $\Gamma$ acts by isometries, it follows that 
\begin{equation*}
\sup_{\tilde q\in[q],\tilde w\in[w]}\langle \tilde q, \tilde w\rangle=\sup_{\lambda\in\Gamma,\gamma\in\Gamma}\langle q*\lambda,w*\gamma\rangle=\sup_{\gamma\in\Gamma}\langle q,w*\gamma\rangle=\sup_{\gamma\in\tilde\Gamma}\langle q,w*\gamma\rangle.
\end{equation*}
Let $\gamma\in\Gamma$ be arbitrary. Then
\begin{equation*}
\langle q,w*\gamma\rangle=\int_Iq\cdot(w*\gamma)=\int_I(q(t)\cdot w_0)\sqrt{\gamma'(t)}dt
=\int_A(q(t)\cdot w_0)\sqrt{\gamma'(t)}+\int_B(q(t)\cdot w_0)\sqrt{\gamma'(t)}
\end{equation*}
The integral over $B$ is clearly bounded above by $0$, since $q(t)\cdot w_0\leq 0$ for $t\in B$. The integral over $A$ can be bounded using the Cauchy Schwarz inequality to yield:

\begin{equation*}
\leq
\sqrt{\int_A(q(t)\cdot w_0)^2dt}\sqrt{\int_A\gamma'(t)dt}\leq\sqrt{\int_I(q(t)\cdot w_0)^2dt},
\end{equation*}
where the last step uses the fact that $\int_A\gamma'(t)dt\leq\int_I\gamma'(t)dt=1$. A straightforward calculation shows that the upper bound is actually achieved by the $\gamma$ defined in the statement of the lemma. 

\end{proof}

We will need a slightly altered form of Lemma \ref{BasicConstantMatch}. We state it as a corollary.

\begin{corollary}\label{RefinedBasicConstantMatch}
Suppose $H$, $J$, and $K$ are finite closed intervals in $\reals$; denote by $L(J)$ the length of $J$. Let $q\in L^2(H,\reals^N)$, and let $w:J\to\reals^N$ be a constant map, $w(t)=w_0$. Then
\begin{equation*}
\sup_{\lambda,\gamma}\langle q*\lambda, w*\gamma\rangle=
\sqrt{\int_A(q(t)\cdot w_0)^2dt}\sqrt{L(J)}
\end{equation*}
where the supremum is taken over all $\lambda:K\to H$ and $\gamma:K\to J$ such that $\lambda$ and $\gamma$ are absolutely continuous, onto, and weakly increasing, and $A$ is the subset of $H$ on which the function $q(t)\cdot w_0$ is non-negative. Furthermore, this supremum is actually realized by an appropriate choice of $\lambda$ and $\gamma$.
\end{corollary}
\begin{proof}
We omit the details of this routine proof. The idea is just to transform each of the three intervals into $I$ using linear bijections, and then use integration by substitution and apply Lemma \ref{BasicConstantMatch}.

\end{proof}
\begin{thm}\label{StepNonStep}
Let $[q]\in{\mathcal S}(I,\reals^N)$ and $[w]\in{\mathcal S}_{st}(I,\reals^N)$. Then there exist $\tilde q\in[q]$ and $\tilde w\in[w]$ such that $d(\tilde q,\tilde w)=d([q],[w])$.
\end{thm}
\begin{proof}
Assume our orbit representatives $q$ and $w$ correspond to unit speed parametrized curves, so that $|q(t)|$ and $|w(t)|$ are constant  (a.e.) in $I$. (These representatives exist by Theorem \ref{arclengthpar}.) It follows from this that $w$ is a step function. This is because, since $[w]\in{\mathcal S}_{st}(I,\reals^N)$, $w$ can only assume a finite sequence of different directions. Since its magnitude is constrained to be constant everywhere, it follows that $w$ assumes only a finite sequence of different values. Hence, there exists a finite sequence of real numbers $0=s_0<s_1<\dots<s_k=1$ and corresponding finite sequence $w_1,w_2,\dots,w_k$ of vectors in $\reals^N$ such that for all $j\in\{1,\dots,k\}$, $w(t)=w_j$ for all $t\in(s_{j-1},s_j)$. Define a set $T\in\reals^{k+1}$ by 
$T=\{(t_0,\dots,t_{k}):0=t_0\leq t_1\leq t_2\leq\dots\leq t_{k}=1\}$. Clearly $T$ is compact. Now define a function $M:T\to \reals$ by 
$$M(t_0,\dots,t_k)=\sum_{j=1}^k\sqrt{\int_{t_{j-1}}^{t_{j}}(q(t)\cdot w_j)^2 dt}\sqrt{s_j-s_{j-1}}.$$ This function is obviously continuous as a function of $(t_0,\dots,t_k)$; since $T$ is compact, $M$ attains a maximum at some element $(\tilde t_0,\dots,\tilde t_k)\in T$. Let $\tilde M=M(\tilde t_0,\dots,\tilde t_k)$. We will show first that there exists $\tilde q\in[q]$ and $\tilde w\in[w]$ such that $\tilde M=\langle \tilde q,\tilde w\rangle$; we will then show that it is the maximum possible value of all such inner products. This will complete the proof of the theorem.

Fix $j\in\{1,\dots,k\}$. In order to apply Corollary \ref{RefinedBasicConstantMatch}, let $H=[\tilde t_{j-1},\tilde t_j]$, $J=[s_{j-1},s_j]$, and $K=[\frac{j-1}{k},\frac{j}{k}]$.  Then Corollary \ref{RefinedBasicConstantMatch} tells us that there exist $\lambda_j:K\to H$ and $\gamma_j:K\to J$ (where $\lambda_j$ and $\gamma_j$ are onto and absolutely continuous) such that
$$\int_K ((q|H)*\lambda_j)(u)\cdot ((w|J)*\gamma_j)(u) du=\sqrt{\int_{t_{j-1}}^{t_{j}}(q(t)\cdot w_j)^2 dt}\sqrt{s_j-s_{j-1}},$$ and that this integral is the maximum possible over all such $\lambda_j$ and $\gamma_j$. Since for each $j$, $\lambda_j(\frac{j-1}{k})=\lambda_{j-1}(\frac{j-1}{k})$ and $\gamma_j(\frac{j-1}{k})=\gamma_{j-1}(\frac{j-1}{k})$, it follows that we can glue together the $\lambda_j$'s to form a single $\tilde\lambda:I\to I$ and can also glue together the $\gamma_j$'s to form a single $\tilde\gamma:I\to I$ such that $\langle q*\tilde\lambda, w*\tilde\gamma\rangle=\tilde M$. This shows that $\tilde M$ is realized as an inner product of a pair of orbit representatives. 

We now show that $\tilde M$ gives the maximum value of the inner product, for all orbit representatives. Suppose $\tilde q\in[q]$ and $\tilde w\in[w]$. By Corollary \ref{generalRnorbitstructure}, there exist $\lambda\in\tilde\Gamma$ and $\gamma\in\tilde\Gamma$ such that $\tilde q=q*\lambda$ and $\tilde w=w*\gamma$. Since $\gamma$ is onto, for each $j$ we can choose $u_j\in\gamma^{-1}(s_j)$, and let $t_j=\lambda(u_j)$. It follows that for each $j$, $\lambda([u_{j-1},u_j])=[t_{j-1},t_j]$, and $\gamma([u_{j-1},u_j])=[s_{j-1},s_j]$. By Corollary \ref{RefinedBasicConstantMatch}, we may conclude that 
$$\int_{u_{j-1}}^{u_j} ((q|[t_{j-1},t_j])*\lambda)(u)\cdot ((w|[s_{j-1},s_j])*\gamma)(u) du\leq\sqrt{\int_{t_{j-1}}^{t_{j}}(q(t)\cdot w_j)^2 dt}\sqrt{s_j-s_{j-1}}.$$
Summing over all $j$ then gives
$$\int_0^1(q*\lambda)(u)\cdot(w*\gamma)(u)du\leq M(t_0,\dots,t_k)\leq\tilde M.$$ Since maximizing the $L^2$ inner product is the same as minimizing the distance, this completes the proof of the theorem.

\end{proof}

\begin{thm}\label{stepstepstep}
If $q,w\in L^2(I,\reals^N)$ are both step functions, then there exist piecewise linear functions $\lambda,\gamma\in\tilde\Gamma$ such that 
$$\langle q*\lambda,w*\gamma\rangle=\sup_{\tilde q\in[q],\tilde w\in[w]}\langle \tilde q, \tilde w\rangle.$$
\end{thm}
\begin{proof} In the statement of Lemma \ref{BasicConstantMatch}, note that if $q$ is a step function, then the function we integrate to get $\gamma$ is also a step function. It follows that $\gamma$ is piecewise linear. In Corollary \ref{RefinedBasicConstantMatch} (still assuming that $q$ is a step function), the reparametrizing functions are obtained from the ones in Lemma \ref{BasicConstantMatch} by composing with linear functions; hence the reparametrizing functions are still piecewise linear. Finally, in Theorem \ref{StepNonStep}, the optimal reparametrizing functions are obtained by gluing together reparametrizing functions of the type formed in Corollary \ref{RefinedBasicConstantMatch}; gluing together piecewise linear functions results in more piecewise linear functions.

\end{proof}  


\section{Preliminaries on Finding a Precise Optimal Matching For Piecewise Linear Functions}

Let $f_1$ and $f_2$ be two continuous, piecewise linear functions $I\to\reals^N$ and let $q_1,q_2\in L^2(I,\reals^N)$ be their SRVFs. We will develop an algorithm which will produce a pair of optimal representatives for $[q_1]$ and $[q_2]$, i.e., $L^2$ functions $\tilde q_1\in[q_1]$ and $\tilde q_2\in[q_2]$ such that $d(\tilde q_1,\tilde q_2)=d([q_1],[q_2])$. Assume that $q_1$ and $q_2$ have the property that the set on which each of them vanishes has measure 0. (If this is not true, than we can replace them by elements of $[q_1]$ and $[q_2]$ that have this property, using Theorem \ref{arclengthpar}.) According to Theorem \ref{generalRnorbitstructure}, these optimal representatives will be of the form $\tilde q_1=q_1*\gamma_1$ and $\tilde q_2 =q_2*\gamma_2$, where $\gamma_1,\gamma_2\in\tilde\Gamma$. We call such a pair $(\gamma_1,\gamma_2)$ an {\it optimal matching} for $f_1, f_2$ (or for $q_1,q_2$).

Since $f_1$ and $f_2$ are piecewise linear, we know that there are subdivisions $0=s_0<s_1<\dots<s_m=1$ and $0=t_0<t_1<\dots<t_n=1$ such that $f_1$ is linear on each subinterval $[s_{i-1},s_i]$ and $f_2$ is linear on each subinterval $[t_{j-1},t_j]$. As a result, we know that $q_1$ is constant on each open interval $(s_{i-1},s_i)$ and $q_2$ is constant on each $(t_{j-1},t_j)$. In general, $q_1$ and $q_2$ are not defined on the endpoints of these intervals, since $f_1$ and $f_2$ are not differentiable at these endpoints. For each $i=1,\dots,m$, let $u_i=q_1((s_{i-1},s_i))$ and for each $j=1,\dots,n$, let $v_j=q_2((t_{j-1},t_j))$. We then define an $n\times m$ matrix $W$, called the {\it weight matrix}, by $W_{ij}=u_i\cdot v_j$. (The dot product here is the ordinary inner product in $\reals^N$.) 

A {\it matching} of $f_1$ and $f_2$ is any pair of reparametrizations $\gamma_1,\gamma_2\in\tilde\Gamma$. Such a pair represents a matching in the sense that for each $z\in I$, the point $f_1(\gamma_1(z))$ on the curve parametrized by $f_1$ is ``matched" to the point $f_2(\gamma_2(z))$ on the curve parametrized by $f_2$. Note that because $\gamma_1$ and $\gamma_2$ are only weakly increasing, this matching does not give a 1-1 correspondence between the points on these two curves. We can assemble $\gamma_1$ and $\gamma_2$ into a single function $\gamma:I\to I\times I$ defined by $\gamma(z) =(\gamma_1(z),\gamma_2(z))$. This function can be thought of as a parametrized curve in $I\times I$ that starts at $(0,0)$ and ends at $(1,1)$. Because $\gamma_1$ and $\gamma_2$ are weakly increasing, this curve can only move vertically upward, horizontally to the right, or in some diagonal direction towards the upper right. We define a {\it vertex} of $I\times I$ to be a point of the form $(s_i,t_j)$, a {\it horizontal gridline} to be a line of the form $t=t_j$ ($i=0,1,\dots,n$), and a {\it vertical gridline} to be a line of the form $s=s_i$ ($j=0,1,\dots,m$). We define the $ij$-block, $G_{ij}$, by $G_{ij}=[s_{i-1},s_i]\times[t_{j-1},t_j]$. Because of the weakly increasing nature of $\gamma_1$ and $\gamma_2$, it is clear that $\gamma^{-1}(G_{ij})$ is always a closed subinterval of $I$. If $\gamma$ is linear and non-constant on an interval $[a,b]$, we define the {\it slope} of $\gamma$ on this interval to be the value of $\gamma_2'/\gamma_1'$. On any such interval, this slope will be $0$, positive, or $\infty$.

\begin{figure}[h]
\begin{center}
\includegraphics[scale=0.45]{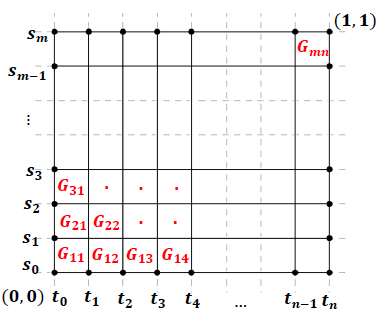}
\caption{The grid}
\label{fig1}
\end{center}
\end{figure}

Given a matching $\gamma$, we define a {\it segment} of $\gamma$ to be the restriction of $\gamma$ to some closed subinterval of $I$. We now define two specific types of segment. 

{\bf Definition of {\it P-segment}:} (Note that this is a long definition! It includes all the statements up until the definition of an N-segment.)  A P-segment is a restriction of $\gamma$ to an interval $[a,b]\subset I$, which has the following properties:

\begin{enumerate}
\item{} $\gamma|_{[a,b]}$ is piecewise linear and injective.
\item{} $\gamma(a)=(s_{i_0-1},t_{j_0-1})$ and $\gamma(b)=(s_{i_1},t_{j_1})$ are vertices, with $i_0\leq i_1$ and $j_0\leq j_1$, but for all $z\in(a,b)$, $\gamma(z)$ is not a vertex.  Furthermore, $W_{i_0,j_0}>0$ and $W_{i_1,j_1}>0$.

\item{} For all blocks $G_{ij}$ such that $\gamma^{-1}(G_{ij})\subset[a,b]$, the restriction of $\gamma$ to $\gamma^{-1}(G_{ij})$ is linear. We define $H_{i,j}$ to be the slope of the segment as it passes through $G_{i,j}$.
\item{} Suppose $\gamma^{-1}(G_{ij})=[c,d]\subset[a,b]$, where $c<d$. If $W_{ij}\leq 0$, then either $H_{i,j}=0$ or $H_{i,j}=\infty$. Visually, this says that the parametrized path $\gamma$ is either vertical or horizontal as it traverses $G_{ij}$. More precisely, if $\gamma$ enters such a $G_{i,j}$ through the left hand vertical edge, then $H_{ij}=0$, while if $\gamma$ enters such a $G_{ij}$ through the lower horizontal edge, then $H_{ij}=\infty$.  If $W_{ij}>0$, then $H_{i,j}$  is not equal to either 0 or $\infty$.

\end{enumerate}

To understand the remaining properties required of a P-segment, note that it begins at the vertex $(s_{i_0-1},t_{j_0-1})$ and passes through the block $G_{i_0,j_0}$ in a linear fashion with slope $H_{i_0,j_0}$, which is equal neither to 0 nor to $\infty$ by the previous items.
Up to reparametrization, the remaining portion of the P-segment is completely determined by the initial vertex $(s_{i_0-1},t_{j_0-1})$, and the initial slope as the segment passes through $G_{i_0,j_0}$. To understand this determination we will describe how the slope $H_{i,j}$ is required to change as the P-segment passes through a gridline from one block to another. First, suppose the segment passes through a vertical gridline from $G_{i,j}$ to $G_{i+1,j}$. There are then three cases to consider:
\begin{enumerate} 
\item{} Both $W_{i,j}$ and $W_{i+1,j}$ are greater than 0. Then the slopes are related as follows:
\begin{equation}
\frac{H_{i+1,j}}{H_{i,j}} = \left(\frac{W_{i+1,j}}{W_{i,j}}\right)^2 
\end{equation}
\item{} $W_{i+1,j}\leq 0$. Then $H_{i+1,j}=0$.
\item{} $W_{i,j}\leq 0$ while $W_{i+1,j}> 0$. By one of the above conditions, we know that $H_{i,j}=0$. To determine $H_{i+1,j}$, we must find the largest value of $k\leq i$ for which for which $W_{k,j}>0$. (This corresponds to the last block $G_{k,j}$ that the segment passed through with non-zero slope.) Using this value of $k$, $H_{i+1,j}$ must then satisfy
\begin{equation}
\frac{H_{i+1,j}}{H_{k,j}} = \left(\frac{W_{i+1,j}}{W_{k,j}}\right)^2. 
\end{equation}

\end{enumerate}
Now, suppose the segment passes through a horizontal gridline from $G_{i,j}$ to $G_{i,j+1}$. The three cases are completely analogous to the cases of the vertical gridline:

\begin{enumerate} 
\item{} Both $W_{i,j}$ and $W_{i,j+1}$ are greater than 0. Then the slopes are related as follows:
\begin{equation}
\frac{H_{i,j+1}}{H_{i,j}} = \left(\frac{W_{i,j}}{W_{i,j+1}}\right)^2 
\end{equation}
\item{} $W_{i,j+1}\leq 0$. Then $H_{i,j+1}=\infty$.
\item{} $W_{i,j}\leq 0$ while $W_{i,j+1}> 0$. By one of the above conditions, we know that $H_{i,j}=\infty$. To determine $H_{i,j+1}$, we must find the largest value of $k\leq j$ for which for which $W_{i,k}>0$. (This corresponds to the last block $G_{i,k}$ that the segment passed through with non-infinite slope.) Using this value of $k$, $H_{i,j+1}$ must then satisfy
\begin{equation}
\frac{H_{i,j+1}}{H_{i,k}} = \left(\frac{W_{i,k}}{W_{i,j+1}}\right)^2. 
\end{equation}

\end{enumerate}
This concludes the definition of a P-segment!

\begin{figure}[h]
\begin{center}
\includegraphics[scale=0.45]{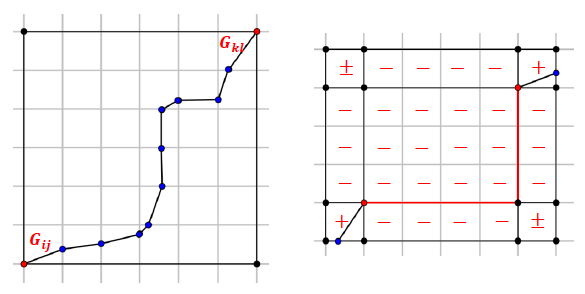}
\caption{Left: P-segment.  Right: N-segment(in red).}
\label{fig1}
\end{center}
\end{figure}

{\bf Definition of {\it N-segment}:}  An N-segment is a restriction of $\gamma$ to an interval $[a,b]\subset I$, which has the following three properties:
\begin{enumerate}
\item{} $\gamma(a)=(s_{i_0},t_{j_0})$ and $\gamma(b)=(s_{i_1},t_{j_1})$ are both vertices, with $i_0\leq i_1$ and $j_0\leq j_1$. 
\item{} The restriction of $\gamma$ to $\left[a,\frac{a+b}{2}\right]$ is linear and runs horizontally from $(s_{i_0},t_{j_0})$ to $(s_{i_1},t_{j_0})$, while the restriction of $\gamma$ to $\left[\frac{a+b}{2},b\right]$ is also linear and runs vertically from $(s_{i_1},t_{j_0})$ to $(s_{i_1},t_{j_1})$. For the special cases in which $i_0=i_1$ or $j_0=j_1$, the entire N-segment is either vertical or horizontal, respectively.
\item{} For $\gamma|_{[a,b]}$ to be an N-segment, there are also the following requirements on certain weights: if $i\in\{i_0+1,\dots,i_1\}$ and $j\in\{ j_0,\dots,j_1+1\}$, then $W_{i,j}\leq 0$. Also, if $i\in\{i_0,\dots,i_1+1\}$ and $j\in\{ j_0+1\dots,j_1\}$, then $W_{i,j}\leq 0$.

\end{enumerate}

\section{Statement and Proof of Main Theorem}
In this section, we state and prove our main result on a canonical form for optimal matchings between piecewise linear curves. 

\begin{thm}\label{MainTheorem}
Let $f_1$ and $f_2$ be piecewise linear functions $I\to\reals^N$. Then there exists an optimal matching $\gamma=(\gamma_1,\gamma_2)$ that has the following properties:
\begin{enumerate}
\item{} $\gamma$ is a sequence of P-segments and N-segments; i.e., there exists a partition $\{0=u_0<u_1<\dots<u_k=1\}$ such that for each $i=1,\dots, k$, $\gamma|_{[u_{i-1},u_{i}]}$ is either a P-segment or an N-segment. 

\item{}$ \gamma$ does not contain two consecutive N-segments.

\item{} Suppose that $\gamma|_{[u_1,u_2]}$ and $\gamma|_{[u_3,u_4]}$ are both P-segments, and suppose that either $u_2=u_3$ or $\gamma|_{[u_2,u_3]}$ is an N-segment. Define $(i_1,j_1)$ and $(i_2,j_2)$ by $(s_{i_1},t_{j_1})=\gamma(u_2)$ and $(s_{i_2},t_{j_2})=\gamma(u_3)$ . Then the final slope of  $\gamma|_{[u_1,u_2]}$ and the initial slope of $\gamma|_{[u_3,u_4]}$ must be related as follows. Let $A=W_{i_1,j_1}$, $B=W_{i_2+1,j_2+1}$, $C=W_{i_1,j_2+1}$, and $D=W_{i_2+1,j_1}$. Then $H_{i_2+1,j_2+1}=\mu^2 H_{i_1,j_1}$ where  

\item{}
\begin{align}
\mu \in \left\{ \begin{array}{ll}
                   \left[ \frac{D^{2}}{AB},\frac{AB}{C^{2}} \right] & \mbox{, if $C>0 , D>0$ } \\                    
 				   \left[ 0, \frac{AB}{C^{2}} \right] & \mbox{, if $D\leq 0 , C>0$ } \\ 				   
                   \left[ \frac{D^{2}}{AB} , \infty \right] & \mbox{, if $D>0 , C \leq 0$ } \\                  
                   \left[ 0,\infty \right] & \mbox{, if $D \leq 0 , C \leq 0$ }
                \end{array} \right. \label{eq12} 
\end{align}
Note that the prescribed $\mu$-interval is empty if $CD>AB$. In that case, there cannot be an optimal matching with one P-segment ending at $(s_{i_1},t_{j_1})$ and the next P-segment beginning at $(s_{i_2},t_{j_2})$
\end{enumerate}

\end{thm}
\begin{proof} We know by Theorem \ref{stepstepstep} that there exists a piecewise linear optimal matching between $f_1$ and $f_2$. Choose such an optimal matching and call it $\gamma=(\gamma_1,\gamma_2):I\to I\times I$. We may assume that $\gamma$ is injective by replacing it by a constant speed reparametrization. Let $V_0,V_1,\dots, V_M$ be an ordered list of all the vertices through which $\gamma$ passes, starting with $V_0=(0,0)$ and ending with $V_M=(1,1)$. From this list, choose an arbitrary vertex $V_i$ (with $i\not\in\{0,M\}$). If either the portion of $\gamma$ from $V_{i-1}$ to $V_i$, or the portion of $\gamma$ from $V_i$ to $V_{i+1}$ passes through a point in the interior of some block $G_{k,l}$ with weight $W_{k,l}>0$, then retain $V_i$ in the list. If neither of these portions of $\gamma$ pass through such a point, then drop $V_i$ from the list. Continue this elimination process until no more vertices can be dropped. Renumber the remaining vertices and revise the number $M$ to reflect the number of vertices remaining in the list. The remaining vertices now have the property that for each $i=1,\dots,M-1$, either the segment of $\gamma$ from $V_{i-1}$ to $V_i$, or the segment from $V_i$ to $V_{i+1}$ passes through at least one point in the interior of some block $G_{k,l}$ with weight $W_{k,l}>0$.

For each $i=0,\dots,M-1$, consider the segment of $\gamma$ from $V_i$ to $V_{i+1}$. There are two possibilities:
\begin{enumerate}
\item{}Type I: If this segment of $\gamma$ passes through a point in the interior of some block $G_{k,l}$ with weight $W_{k,l}>0$, we will prove that it can be replaced by a P-segment without affecting the optimality, i.e., without affecting the value of $\int_0^1(q_1*\gamma_1)(u)\cdot (q_2*\gamma_2)(u)du$.
\item{}Type II: If this segment of $\gamma$ does not pass through such a point, then we will prove that it can be replaced by an N-segment, without affecting the optimality, i.e., without affecting the value of $\int_0^1(q_1*\gamma_1)(u)\cdot (q_2*\gamma_2)(u)du$.
\end{enumerate}
\begin{lemma}\label{LinearCalc}
Let $v,w\in\reals^N$ be two vectors, and define two constant functions $q_1:[a,b]\to\reals^N$ and $q_2:[c,d]\to\reals^N$ by $q_1(s)=v$ and $q_2(t)=w$. Let $\alpha<\beta$; define $\gamma_1:[\alpha,\beta]\to[a,b]$ to be the unique linear function such that $\gamma_1(\alpha)=a$ and $\gamma_1(\beta)=b$ and define $\gamma_2:[\alpha,\beta]\to[c,d]$ to be the unique linear function such that $\gamma_2(\alpha)=c$ and $\gamma_2(\beta)=d$. Then $\int_\alpha^\beta(q_1*\gamma_1)(u)\cdot(q_2*\gamma_2)(u) du=(v\cdot w)\sqrt{b-a}\sqrt{d-c}$.
\end{lemma}
\begin{proof} This is an easy calculation since $(q_1*\gamma_1)(u)=v\sqrt{(b-a)/(\beta-\alpha)}$ and $(q_2*\gamma_2)(u)=w\sqrt{(d-c)/(\beta-\alpha)}$ are constant functions!
\end{proof}
\begin{lemma}\label{LinearOptimal}
Let $v,w\in\reals^N$ be two vectors satisfying $v\cdot w>0$ and define two constant functions $q_1:[a,b]\to\reals^N$ and $q_2:[c,d]\to\reals^N$ by $q_1(s)=v$ and $q_2(t)=w$. 
Let $\gamma_1:[\alpha,\beta]\to[a,b]$ and  $\gamma_2:[\alpha,\beta]\to[c,d]$ be surjective absolutely continuous functions with both $\gamma_1'(u)>0$ and $\gamma_2'(u)>0$ almost everywhere for $u\in[\alpha,\beta]$. 
Let $\tilde\gamma_1$ and  $\tilde\gamma_2$ be the unique linear bijections $[\alpha,\beta]\to[a,b]$ and  $[\alpha,\beta]\to[c,d]$, respectively.

Then 
$$\int_\alpha^\beta (q_1*\gamma_1)(u)\cdot (q_2*\gamma_2)(u)du\leq\int_\alpha^\beta (q_1*\tilde\gamma_1)(u)\cdot (q_2*\tilde\gamma_2)(u)du=(v\cdot w)\sqrt{b-a}\sqrt{d-c}.$$
\end{lemma}
\begin{proof}
The main tool here is the Cauchy-Schwarz inequality. Note that for $u\in[\alpha,\beta]$, $(q_1*\gamma_1)(u)=\sqrt{\gamma_1'(u)}v$ and $(q_2*\gamma_2)(u)=\sqrt{\gamma_2'(u)}w$. We then compute:
$$\int_\alpha^\beta (q_1*\gamma_1)(u)\cdot (q_2*\gamma_2)(u)du=\int_\alpha^\beta v\cdot w\sqrt{\gamma_1'(u)}\sqrt{\gamma_2'(u)}du
=v\cdot w\int_\alpha^\beta \sqrt{\gamma_1'(u)}\sqrt{\gamma_2'(u)}du$$
$$\leq v\cdot w\sqrt{\int_\alpha^\beta\gamma_1'(u)du}\sqrt{\int_\alpha^\beta\gamma_2'(u)du}=v\cdot w\sqrt{\gamma_1(\beta)-\gamma_1(\alpha)}\sqrt{\gamma_2(\beta)-\gamma_2(\alpha)}$$
$$=v\cdot w\sqrt{b-a}\sqrt{d-c}$$
where the inequality is just the Cauchy-Schwarz inequality. Finally, note that if we replace each $\gamma_i$ by $\tilde\gamma_i$ for $i=1,2$, then since each $\tilde\gamma_i'$ is a positive constant function, the Cauchy-Schwarz inequality is actually an equality.
\end{proof}
\begin{lemma}\label{OptimalPositiveBlock}
Suppose $\gamma$ passes through a point in the interior of $G_{k,l}$ for which $W_{k,l}>0$. It follows that $\gamma^{-1}(G_{k,l})=[\alpha,\beta]$, where $\alpha<\beta$. If we replace $\gamma|_{[\alpha,\beta]}$ by the unique linear map $\tilde\gamma:[\alpha,\beta]\to G_{k,l}$ that agrees with $\gamma$ at $\alpha$ and $\beta$, then 
$$\int_\alpha^\beta(q_1*\tilde\gamma_1)(u)\cdot (q_2*\tilde\gamma_2)(u)du\geq\int_\alpha^\beta(q_1*\gamma_1)(u)\cdot (q_2*\gamma_2)(u)du.$$ 
Since we are assuming that $\gamma$ is optimal, it follows that this inequality is actually an equality, so we can replace $\gamma|_{[\alpha,\beta]}$ by the linear map $\tilde\gamma$ without affecting its optimality.
\end{lemma}
\begin{proof} Let $a=\gamma_1(\alpha)$, $b=\gamma_1(\beta)$, $c=\gamma_2(\alpha)$, and $d=\gamma_2(\beta)$. Since $\gamma([\alpha,\beta])\subset G_{k,l}$, it follows that $q_1=v$ is constant on $[a,b]$ and $q_2=w$ is constant on $[c,d]$. Also, since $W_{k,l}>0$, we know that $v\cdot w>0$. Then, from Theorem \ref{LinearOptimal} it follows immediately that 
$$\int_\alpha^\beta(q_1*\tilde\gamma_1)(u)\cdot (q_2*\tilde\gamma_2)(u)du\geq\int_\alpha^\beta(q_1*\gamma_1)(u)\cdot (q_2*\gamma_2)(u)du.$$
The rest of the Lemma follows from this.

There is one other small point to consider here; in our proof, we tacitly assumed that $q_1$ and $q_2$ are defined on all of $\gamma_1([\alpha,\beta])$ and $\gamma_2([\alpha,\beta])$, respectively. However, in our case either $q_1$ or $q_2$ will fail to be defined at points along the boundary of the block. As a result, one should consider separately the possibility of a $\gamma$ that stays along the edge of $G_{k,l}$ for either an initial portion or a final portion of $[\alpha,\beta]$. However it is not possible for such an $\gamma$ to achieve a higher value for the integral in question. The reason is that the contribution of the integral along the edge of $G_{k,l}$ will always be zero (since in these regions either $\gamma_1'$ or $\gamma_2'$ will vanish). And in the remainder of the integral corresponding to such a $\gamma$, the value of $b-a$ and/or the value of $d-c$ will have to be reduced, which will result in a reduction of the maximum value of the integral as given in Lemma \ref{LinearCalc}.

\end{proof}
\begin{lemma}\label{OptimalNegativeBlock}
Suppose $\gamma$ passes through a point in the interior of $G_{k,l}$ for which $W_{k,l}\leq0$. It follows that $\gamma^{-1}(G_{k,l})=[\alpha,\beta]$, where $\alpha<\beta$. If we replace $\gamma|_{[\alpha,\beta]}$ by a continuous piecewise linear $\tilde\gamma$ that agrees with $\gamma$ on $\alpha$ and $\beta$ but is made up of a finite sequence of vertical (upwards) and horizontal (to the right) segments, then the resulting $\gamma$ will still be optimal.

\end{lemma}
\begin{proof}
Because $\gamma([\alpha,\beta])\subset G_{k,l}$, it follows that 
$$\int_\alpha^\beta (q_1*\gamma_1)(u)\cdot(q_2*\gamma_2)(u) du=\int_\alpha^\beta (v_k)\cdot(w_l)\sqrt{\gamma_1'(u)}\sqrt{\gamma_2'(u)}du
\leq 0,$$
since we are assuming that $(v_k)\cdot(w_l)=W_{k,l}\leq 0$. However, note that 
$$\int_\alpha^\beta (q_1*\tilde\gamma_1)(u)\cdot(q_2*\tilde\gamma_2)(u)du=0,$$ 
since for all $u\in[\alpha,\beta]$, either $\gamma_1'(u)=0$ or $\gamma_2'(u)=0$. Since $\gamma$ is assumed to be optimal, it follows that 
$\int_\alpha^\beta (q_1*\gamma_1)(u)\cdot(q_2*\gamma_2)(u) du=0$, and this contribution doesn't change if we replace $\gamma|_{[\alpha,\beta]}$ by $\tilde\gamma$.

\end{proof}
Now, suppose we have an optimal matching $\gamma$, and within that $\gamma$ we have chosen a segment, $\gamma|_{[a,b]}$, of Type I. We have proved that we can replace this segment of $\gamma$ with an equally optimal segment that is linear each time it passes through a block $G_{k,l}$ for which $W_{k,l}>0$, and that is a finite sequence of horizontal and vertical segments each time it passes through a block $G_{k,l}$ for which $W_{k,l}\leq0$. So assume $\gamma|_{[a,b]}$ has these properties. We claim that there is at least one $G_{k,l}$, with $W_{k,l}>0$, that our segment passes through with positive, non-infinite slope. To prove this claim, note that if no such $W_{j,k}$ exists, then $\int_a^b (q_1*\gamma_1)(u)\cdot (q_2*\gamma_2)(u)du=0$. But then, replacing $\gamma|_{[a,b]}$ by a path that uses a sequence of horizontal and vertical segments to get from $\gamma(a)$ to $(s_{j-1},t_{k-1})$, then a diagonal line from $(s_{j-1},t_{k-1})$ to $(s_{j},t_{k})$, and then a sequence of horizontal and vertical segments to get from $(s_{j},t_{k})$ to $\gamma(b)$, would result in a positive integral over this segment, contradicting optimality.

Thus, choose a block $G_{k,l}$, with $W_{k,l}>0$, that our segment passes through with positive, non-infinite slope. If our entire segment $\gamma|_{[a,b]}$ passes from the lower left vertex of this block to the upper right vertex, then $\gamma|_{[a,b]}$ is a diagonal line joining these vertices, proving it is a P-segment. So, assume $\gamma|_{[a,b]}$ either enters or exits $G_{k,l}$ through a point on an edge that is not a vertex. Just to be specific, assume that $\gamma|_{[a,b]}$ exits $G_{k,l}$ through a point on its right edge, which would be of the form $(s_k,t^*)$, where $t_{l-1}<t^*<t_l$. Our next task to to examine what happens to $\gamma|_{[a,b]}$ as it passes through the next block to the right, $G_{k+1,l}$. First, consider the case in which $W_{k+1,l}>0$. In that case, by Lemma \ref{OptimalPositiveBlock}, we know $\gamma$ is linear as it passes through $G_{k+1,l}$. The following Lemma tells us the relationship between the slopes $H_{k,l}$ and $H_{k+1,l}$ as $\gamma$ passes through these blocks.
\begin{lemma}\label{AdjacentPositiveSlopes}

\begin{enumerate}
\item{} Assume that the adjacent blocks $G_{k,l}$ and $G_{k+1,l}$ both have positive weights, and suppose that an optimal $\gamma$ passes from $G_{k,l}$ to $G_{k+1,l}$ at the point $(s_k,t^*)$, where $t_{l-1}<t^*<t_l$. Furthermore, assume that $\gamma$ has positive and non-infinite slope in at least one of these two adjacent blocks. Then the slope of $\gamma$ in the other block is also positive and non-infinite, and these two slopes are related by 
\begin{equation}
\frac{H_{k+1,l}}{H_{k,l}} = \left(\frac{W_{k+1,l}}{W_{k,l}}\right)^2 
\end{equation}
\item{} Assume that the adjacent blocks $G_{k,l}$ and $G_{k,l+1}$ both have positive weights, and suppose that an optimal $\gamma$ passes from $G_{k,l}$ to $G_{k,l+1}$ at the point $(s^*,t_l)$, where $s_{k-1}<s^*<s_k$. Furthermore, assume that $\gamma$ has positive and non-infinite slope in at least one of these two adjacent blocks. Then the slope of $\gamma$ in the other block is also positive and non-infinite, and these two slopes are related by 
\begin{equation}
\frac{H_{k,l+1}}{H_{k,l}} = \left(\frac{W_{k,l}}{W_{k,l+1}}\right)^2 
\end{equation}

\end{enumerate}
\end{lemma}
\begin{proof}
We assume that the adjacent blocks $G_{k,l}$ and $G_{k+1,l}$ both have positive weights, and that $\gamma$ passes from $G_{k,l}$ to $G_{k+1,l}$ at the point $(s_k,t^*)$, where $t_{l-1}<t^*<t_l$. Furthermore, we assume that $\gamma$ has positive and nonzero slope in $G_{k,l}$. First, we will show that $\gamma$ must pass through an interior point of $G_{k+1,l}$. If it doesn't, then it would have to follow a vertical path in the left edge of $G_{k+1,l}$, which is the same as the right edge of $G_{k,l}$; but this would violate the fact that is it linear while in $G_{k,l}$. So choose $\alpha$ and $\beta$ such that $\gamma(\alpha)=(\sigma_1,\tau_1)$ is an interior point of $G_{k,l}$ and $\gamma(\beta)=(\sigma_2,\tau_2)$ is an interior point of $G_{k+1,l}$. It follows that 
$$\int_\alpha^\beta (q_1*\gamma_1)(u)\cdot (q_2*\gamma_2)(u)du=W_{k,l}\sqrt{(s_k-\sigma_1)(t^*-\tau_1)}+W_{k+1,l}\sqrt{(\sigma_2-s_k)(\tau_2-t^*)}$$
by Lemma \ref{LinearCalc}. If we view the above formula as a function of a single variable $t^*$, it is an easy Calc I problem to show that the value of the integral is maximized when we choose $t^*$ so that 
$$\frac{(\tau_2-t^*)}{(\sigma_2-s_k)}=\left(\frac{W_{k+1,l}}{W_{k,l}}\right)^2\frac{(t^*-\tau_1)}{(s_k-\sigma_1)}.$$ 
Since we are assuming that $\gamma$ is optimal, it follows that this slope relationship must hold. The other cases of the Lemma follow by analogous arguments.
\end{proof}
Given a Type I matching, we have shown that it must pass through an interior point of a block $G_{k,l}$, of positive weight, with a slope that is neither zero nor infinity. As we follow this segment in either direction, Lemma \ref{AdjacentPositiveSlopes} tells us how the slope of $\gamma$ must change, as long as it enters new blocks of positive weight through non-vertex edge points. (Of course, if it meets a vertex, that terminates our Type I segment.) We now address the question of what happens when a matching passes from a block of positive weight (which it traverses a slope that is neither zero nor infinity) to a block with non-positive weight.
\begin{lemma}\label{PositiveToNegativeSlope}
\begin{enumerate}

\item{} Suppose an optimal matching $\gamma$ passes from a block $G_{k,l}$ with to a block $G_{k+1,l}$ at a point $(s_k,t^*)$, where $t_{l-1}<t^*<t_l$. Assume that one of these blocks has positive weight, and the other has non-positive weight. Also, assume that the slope of $\gamma$ in the block with positive weight is non-zero and non-infinite. Then the slope of $\gamma$ in the block with non-positive weight is zero; hence, $\gamma$ traverses the block with non-positive weight along the horizontal line segment $t=t^*$.

\item{} Suppose an optimal matching $\gamma$ passes from a block $G_{k,l}$  to a block $G_{k,l+1}$  at a point $(s^*,t_l)$, where $s_{k-1}<s^*<s_k$. Assume that one of these blocks has positive weight, and the other has non-positive weight. Also, assume that the slope of $\gamma$ in the block with positive weight is non-zero and non-infinite. Then the slope of $\gamma$ in the block with non-positive weight is infinite; hence, $\gamma$ traverses the block with non-positive weight along the vertical line segment $s=s^*$.

\end{enumerate}
\end{lemma}
\begin{proof} Suppose we are in the first case. Also, to be definite, assume that $W_{k,l}>0$ while $W_{k+1,l}\leq 0$ and that the slope of $\gamma$ in $G_{k,l}$ is non-zero and non-infinite. We proceed by contradiction; suppose that $\gamma$ exits $G_{k+1,l}$ at a point other than $(s_k,t^*)$. In that case, the exit point must be of the form $(\tilde s,\tilde t)$, where $\tilde s>s_k$ and $\tilde t>t^*$.  

By Lemma \ref{OptimalNegativeBlock}, we know that the portion of $\gamma$ passing through $G_{k+1,l}$ will contribute 0 to $\int_a^b (q_1*\gamma_1)(u)\cdot (q_2*\gamma_2)(u)du$. Consider what happens if we replace the portion of $\gamma$ passing through these two blocks by a segment that enters $G_{k,l}$ at the same entry point as $\gamma$, passes linearly through $G_{k,l}$ to the point $(s_k,\tilde t)$,  and then proceeds through $G_{k+1,l}$ by the horizontal segment from $(s_k,\tilde t)$ to $(\tilde s,\tilde t)$. This replacement will increase the integral $\int_a^b (q_1*\gamma_1)(u)\cdot (q_2*\gamma_2)(u)du$, since it will increase the contribution of the portion of $\gamma$ in $G_{k,l}$ (by Lemma \ref{LinearCalc}), while not changing the contribution of the portion in $G_{k+1,l}$, which will still be zero. Thus we contradict the optimality of the original $\gamma$, and the proof of Case (1) of the Lemma is complete. The proof of Case (2) is analogous and we omit it.

\end{proof}

Thus, given a Type I segment, we know it passes through an interior point of a block $G_{k,l}$, of positive weight, with a slope that is neither zero nor infinity. Following this segment in both directions, we know precisely what happens to this segment as if it encounters a block of  positive weight or a block of negative weight. What happens if it encounters several blocks of non-positive weight in a row?
\begin{lemma}
\begin{enumerate}
\item{} Suppose $G_{k,l}$ and $G_{p,l}$ are blocks of positive weight, where $k<p$, and suppose that the intervening blocks $G_{k+1,l}, G_{k+2,l},\dots, G_{p-1,l}$ all have non-positive weights. If $\gamma$ passes through $G_{k,l}$ with non-zero and non-infinite slope, and meets the boundary of $G_{k,l}$ at the point $(s_k,t^*)$, where $t_{l-1}<t^*<t_l$, then $\gamma$ proceeds through all the intervening blocks $G_{k+1,l}, G_{k+2,l},\dots, G_{p-1,l}$ with slope 0 (along the horizontal line $t=t^*$), and then passes through the block $G_{p,l}$ with slope related to the slope in $G_{k,l}$ by the formula
\begin{equation}
\frac{H_{p,l}}{H_{k,l}} = \left(\frac{W_{p,l}}{W_{k,l}}\right)^2 
\end{equation}
If instead of assuming $\gamma$ passes through $G_{k,l}$ with positive, non-infinite slope, we assume that it passes through $G_{p,l}$ with positive, non-infinite slope, then we can again conclude that it passes through the intervening blocks with slope 0 and passes through $G_{k,l}$ with positive, non-infinite slope, and that these slopes are related by the same equation.
\item{}
Suppose $G_{k,l}$ and $G_{k,p}$ are blocks of positive weight, where $l<p$, and suppose that the intervening blocks $G_{k,l+1}, G_{k,l+2},\dots, G_{k,p-1}$ all have non-positive weights. If $\gamma$ passes through $G_{k,l}$ with non-zero and non-infinite slope, and meets the boundary of $G_{k,l}$ at the point $(s^*,t_l)$, where $s_{k-1}<s^*<s_k$, then $\gamma$ proceeds through all the intervening blocks $G_{k,l+1}, G_{k,l+2},\dots, G_{k,p-1}$ with slope $\infty$ (along the vertical line $s=s^*$), and then passes through the block $G_{k,p}$ with slope related to the slope in $G_{k,l}$ by the formula
\begin{equation}
\frac{H_{k,p}}{H_{k,l}} = \left(\frac{W_{k,l}}{W_{k,p}}\right)^2 
\end{equation}
If instead of assuming $\gamma$ passes through $G_{k,l}$ with positive, non-infinite slope, we assume that it passes through $G_{k,p}$ with positive, non-infinite slope, then we can again conclude that it passes through the intervening blocks with slope $\infty$ and passes through $G_{k,l}$ with positive, non-infinite slope, and that these slopes are related by the same equation.

\end{enumerate}

\end{lemma}

Note that this Lemma contains Lemma \ref{AdjacentPositiveSlopes} as the special case in which the number of intervening blocks (with non-positive slopes) is zero.
\begin{proof}
For definiteness, assume we are in Case (1) of the lemma. The proof that $\gamma$ continues with slope 0 through all the intervening blocks with non-positive weights is the same as the proof of Lemma \ref{PositiveToNegativeSlope}; if not, we could replace $\gamma$ by a matching would violate the optimality of the $\gamma$. Now that we know that $\gamma$ has zero slope through the intervening blocks, the proof of the relationship between the slopes in $G_{k,l}$ and $G_{p,l}$ is identical to the proof of the relationship in Lemma \ref{AdjacentPositiveSlopes}, the only modification being that we let $t^*$ represent that height of the horizontal line instead of just the height of the transition point. Case (2) is completely analogous and we omit its proof.
\end{proof}

If we are given a Type I segment, we have shown it passes through an interior point of a block $G_{k,l}$, of positive weight, with a slope that is neither zero nor infinity. Following the segment from this block in each direction, we have now proved that until it encounters a vertex, it must follow the definition of a P-segment. Of course when it encounters a vertex in either direction, that will be the end of the Type I segment. Thus, we have proved that each Type I segment is a P-segment.

We now turn to the proof that each Type II segment can be replaced by an N-segment without altering its optimality. We start with an optimal matching $\gamma$. Assume that our Type II segment is $\gamma_{[a,b]}$. Recall from the definition of a Type II segment, that it starts at a vertex, ends at a vertex, and never passes through an interior point of a block with positive weight. Also, we know that we cannot have two consecutive Type II segments, so if it is preceded by a segment, that segment is now known to be a P-segment, and if it is followed by a segment, that segment is known to be a P-segment.  Let $\gamma(a)=(s_{p-1},t_{q-1})$ and let $\gamma(b)=(s_{k},t_l)$. Because $\gamma_{[a,b]}$ does not pass through an interior point of any block with positive weight, we know that $\int_a^b (q_1*\gamma_1)(u)\cdot (q_2*\gamma_2)(u)du\leq 0$ (since the integrand is non-positive almost everywhere). However, if we replaced $\gamma_{[a,b]}$ by a horizontal segment from $(s_{i-1},t_{j-1})$ to $(s_{k},t_{j-1})$ followed by a vertical segment from $(s_{p},t_{q-1})$ to  $(s_{k},t_l)$, then $\int_a^b (q_1*\gamma_1)(u)\cdot (q_2*\gamma_2)(u)du= 0$; by the optimality of $\gamma$, it follows that for our Type II segment, $\int_a^b (q_1*\gamma_1)(u)\cdot (q_2*\gamma_2)(u)du=0$, and we may make this replacement without affecting the optimality.

\begin{lemma}\label{MostNegativeBlocks}
If $\gamma$ is optimal and $\gamma_{[a,b]}$ is a Type II segment from the vertex $(s_{p-1},t_{q-1})$ to the vertex $(s_{k},t_l)$ then $W_{i,j}\leq 0$ for all $i, j$ satisfying $p\leq i\leq k$ and $q\leq j\leq l$. 
 \end{lemma}
 \begin{proof}
 Suppose not; choose $(i,j)$ such that $p\leq i\leq k$ and $q\leq j\leq l$ but $W_{i,j}> 0$. Then, if we replace $\gamma|_{[a,b]}$ by a segment that starts at $(s_{p-1},t_{q-1})$, then proceeds by first a horizontal segment and then a vertical segment to  $(s_{i-1},t_{j-1})$, then by a linear segment from $(s_{i-1},t_{j-1})$ to $(s_{i},t_{j})$, and then by first a 
horizontal segment and then a vertical segment to $(s_k,t_l)$, we will increase the value of this integral from 0 to a positive number. This contradicts the optimality of $\gamma$, and proves the Lemma.
\end{proof}
To satisfy the definition of N-segment, we need to prove a few more weights are $\leq 0$.
\begin{lemma}
If $\gamma$ is optimal and $\gamma_{[a,b]}$ is a Type II segment from the vertex $(s_{p-1},t_{q-1})$ to the vertex $(s_{k},t_l)$ then $W_{i,j}\leq 0$ for all $i, j$ satisfying any one of the following conditions:
\begin{itemize}
\item{} $p\leq i\leq k$ and $j=q-1$ 
\item{} $p\leq i\leq k$ and $j=l+1$
\item{} $i=p-1$ and $q\leq j\leq l$
\item{} $i=k+1$ and $q\leq j\leq l$
\end{itemize}
\end{lemma}
Note that in some cases one or more of these conditions may be vacuous; for example, if $q=0$, then there is no block $G_{i,j}$ satisfying $j=q-1$.

\begin{proof}
The proof is the same for all four conditions, so consider the first one. Proceed by contradiction; suppose that $W_{i,j}>0$, where $p\leq i\leq k$ and $j=q-1$. Assume that $\gamma|_{[a,b]}$ takes the form of a horizontal segment from $(s_{p-1},t_{q-1})$ to $(s_{k},t_{q-1})$, and then a vertical segment from $(s_{k},t_{q-1})$ to $(s_{k},t_{l})$. (We know that by Lemma \ref{MostNegativeBlocks}, $\int_a^b (q_1*\gamma_1)(u)\cdot (q_2*\gamma_2)(u)du\leq 0$; since the segment described makes the integral equal to zero, it is an optimal one.) Since $q-1>0$ (so $q>1$) in this case, we know that our current Type II segment has a segment preceding it, and we have proved that this preceding segment is a P-segment. We know that $G_{p-1,q-1}$ is the last block that this preceding P-segment passed through, and we also know that because it was a P-segment, $W_{p-1,q-1}>0$, and $H_{p-1,q-1}$ is positive and finite. Let $\alpha<a$ be the lowest parameter value for which $\gamma(\alpha)\in G_{p-1,q-1}$. Let $(\tilde s, \tilde t)=\gamma(\alpha)$. Since the slope of $\gamma$ in $G_{p-1,q-1}$ is positive, we know that $\tilde t<t_{q-1}$. Now, focus attention on the segment of $\gamma$ from $(\tilde s, \tilde t)$ to $(s_i,t_{q-1})$. This segment consists of a straight line segment (of positive slope) from$(\tilde s, \tilde t)$ to $(s_{p-1},t_{q-1})$, followed by a horizontal line segment from $(s_{p-1},t_{q-1})$ to $(s_i,t_{q-1})$. For any $h$ satisfying $\tilde t\leq h\leq t_{q-1}$, define a segment $\gamma_h$ consisting of a straight line from $(\tilde s, \tilde t)$ to $(s_{p-1},h)$, followed by a horizontal line from $(s_{p-1},h)$ to $(s_{i-1},h)$, followed by a straight line from  $(s_{i-1},h)$ to  $(s_i,t_{q-1})$.

By Lemma \ref{LinearCalc}, the contribution of the segment $\gamma_h$ to the integral in question is 
$$C(h)=W_{p-1,q-1}\sqrt{s_{p-1}-\tilde s}\sqrt{h-\tilde t}+W_{i,q-1}\sqrt{s_i-s_{i-1}}\sqrt{t_{q-1}-h}.$$
Note the contribution of the horizontal segment is zero and, by our assumptions, $W_{p-1,q-1}$ and $W_{i,q-1}$ are both greater than zero. Clearly $C(h)$ is continuous for $\tilde t\leq h\leq t_{q-1}$, and is differentiable except at the endpoints of this $h$-interval. When $h=t_{q-1}$, the segment $\gamma_h$ coincides with the segment of $\gamma$ under consideration. Clearly, as $h\to t_{q-1}$, $C'(h)\to -\infty$, since the derivative of $f(x)=\sqrt{x}$ approaches $\infty$ as $x\to 0$. This implies that for values of $h$ within some some small interval $(t_{q-1}-\epsilon,t_{q-1}]$, $C(h)$ is a decreasing function of $h$, and so for $h\in(t_{q-1}-\epsilon,t_{q-1})$, $C(h)>C(t_{q-1})$. This contradicts the optimality of our original $\gamma$, and completes the proof of the Lemma. 
 
\end{proof}
The lemmas we have proved show that a segment of Type I is always a P-segment and a segment of Type II is always an N-segment, establishing Statements (1) and (2) of Theorem \ref{MainTheorem}. What remains is to prove Statement (3) of Theorem \ref{MainTheorem}, which gives a relationship between the final slope of a P-segment, and the initial slope of the next P-segment (whether or not there is an N-segment between them).

First consider the case in which one P-segment of our optimal matching $\gamma$ ends at the vertex $(s_i,t_j)$, and the next one begins at the same point. Since these are P-segments, we already know that their slopes $H_{i,j}$ in $G_{i,j}$ and $H_{i+1,j+1}$ in $G_{i+1,j+1}$ are both positive. Let $\mu=\sqrt{{H_{i+1,j+1}\over H_{i,j}}}$. We need to prove that $\mu$ satisfies the appropriate inequalities given in Statement (3) of Theorem \ref{MainTheorem}. Note that these inequalities depend on the sign of $C$ and $D$. (Because we are dealing with P-segments, $A$ and $B$ must both be positive, by definition.) This argument proceeds by contradiction; we show that if $\mu$ is outside the prescribed intervals, then $\gamma$ is not optimal.

We begin by assuming that $D>0$. In either of the two cases where this holds, the lower end of the prescribed interval for $\mu$ is $D^2/AB$. So, suppose that $\mu<D^2/AB$. Choose a point on $\gamma$ in the interior of $G_{i,j}$. This point will be of the form $(s_i-p,t_j-q)$, where $p,q>0$. Likewise, choose a point on $\gamma$ in the interior of $G_{i+1,j+1}$. This point will be of the form $(s_i+u,t_j+v)$, where $u,v>0$. Now, for arbitrary $x\in[0,u]$ and $y\in[0,q]$, consider a path $\gamma^{x,y}=(\gamma^{x,y}_1,\gamma^{x,y}_2)$, composed of the following three pieces: first, the line segment from $(s_i-p,t_j-q)$ to $(s_i,t_j-y)$; second, the line segment from $(s_i,t_j-y)$ to $(s_i+x,t_j)$; third, the line segment from $(s_i+x,t_j)$ to $(s_i+u,t_j+v)$. Assume that the portion of $\gamma$ from $(s_i-p,t_j-q)$ to $(s_i+u,t_j+v)$ corresponds to the parameter interval $z\in[\alpha,\beta]$. Parameterize $\gamma^{x,y}$ using this same parameter interval, and assume that it is linear on each of the three segments. Define
$$E(x,y)=\int_\alpha^\beta (q_1*\gamma^{x,y}_1)(z)\cdot (q_2*\gamma^{x,y}_2)(z)dz.$$
By applying Lemma \ref{LinearCalc} to the three linear pieces of $\gamma^{x,y}$, we obtain 
$$E(x,y)=A\sqrt{p}\sqrt{q-y}+D\sqrt{x}\sqrt{y}+B\sqrt{v}\sqrt{u-x}$$
where we are in the case of $A,B,D>0$.
It is an easy exercise in two-variable calculus that the function $E(x,y)$ has a unique absolute maximum on the domain $(x,y)\in[0,\infty)\times[0,\infty)$, and that this maximum occurs at the point
$$x_0=u\left(\frac{D^4q-B^2A^2\left(\frac{pv}{u}\right)}{D^4q+D^2B^2v}\right)$$
$$y_0=q\left(\frac{D^4u-B^2A^2\left(\frac{pv}{q}\right)}{D^4u+D^2A^2p}\right)$$
We now observe that this maximum $(x_0,y_0)$ lies in $(0,u)\times(0,q)$, as follows. First, note that every individual variable occurring in the expressions for $x_0$ and $y_0$ has a positive value. Furthermore, recall that 
$\mu=\sqrt{{H_{i+1,j+1}\over H_{i,j}}}=\sqrt{v/u\over q/p}$. Since we are assuming that $\mu<D^2/AB$, it follows immediately that the numerators in the formulae for both $x_0$ and $y_0$ are positive and therefore $x_0,y_0>0$. Since the numerator in the fraction for $x_0$ is less than $D^4q$, while the denominator is greater than $D^4q$, it follows that $x_0<u$ and, similarly, that $y_0<q$. Hence we have shown that $(x_0,y_0)$ lies in $(0,u)\times(0,q)$. Since $E$ has an absolute maximum at $(x_0,y_0)$, it follows that $E(x_0,y_0)>E(0,0)$. But this contradicts the optimality of $\gamma$, since $\gamma^{0,0}$ corresponds exactly to our original $\gamma$!

Similarly, under the assumption that $C>0$, we show that $\mu>AB/C^2$ leads to a contradiction. This proves Statement (3) for two adjacent P-segments.

The case of two P-segments separated by a single N-segment is similar. Suppose one P-segment ends at a vertex $(s_i,t_j)$ and the next one starts at $(s_k,t_l)$, and there is an N-segment from $(s_i,t_j)$ to $(s_k,t_l)$. Once, again, we will assume we are in the case $D>0$, and suppose that $\mu<D^2/AB$. Choose a point on $\gamma$ in the interior of $G_{i,j}$. This point will be of the form $(s_i-p,t_j-q)$, where $p,q>0$. Likewise, choose a point on $\gamma$ in the interior of $G_{k+1,l+1}$. This point will be of the form $(s_k+u,t_l+v)$, where $u,v>0$. Note that the portion of $\gamma$ from $(s_i-p,t_j-q)$ to $(s_k+u,t_l+v)$ consists of four line segments: first from $(s_i-p,t_j-q)$ to $(s_i,t_j)$, second from  $(s_i,t_j)$ to $(s_k,t_j)$, third from $(s_k,t_j)$ to $(s_k,t_l)$, and fourth from $(s_k,t_l)$ to $(s_k+u,t_l+v)$. Now, for arbitrary $x\in[0,u]$ and $y\in[0,q]$, consider a path $\gamma^{x,y}=(\gamma^{x,y}_1,\gamma^{x,y}_2)$, composed of the following five line segments: first from $(s_i-p,t_j-q)$ to $(s_i,t_j-y)$, second from $(s_i,t_j-y)$ to $(s_k,t_j-y)$, third from $(s_k,t_j-y)$ to $(s_k+x,t_j)$, fourth from $(s_k+x,t_j)$ to $(s_k+x,t_l)$, and fifth from $(s_k+x,t_l)$ to $(s_k+u,t_l+v)$. The rest of the argument proceeds just as before; the contribution of the integral over the segment $\gamma^{x,y}$ is again given by the formula
$$E(x,y)=A\sqrt{p}\sqrt{q-y}+D\sqrt{x}\sqrt{y}+B\sqrt{v}\sqrt{u-x}$$
since the horizontal and vertical segments have no contributions. By finding that the maximum value of $E(x,y)$ does not occur at $(x,y)=(0,0)$, we contradict the assumption that $\gamma$ was optimal.

This completes the proof of Theorem \ref{MainTheorem}.

\end{proof}
\section{Algorithm for Producing a Precise Optimal Matching of PL Curves}

In Theorem \ref{MainTheorem}, we proved that given PL curves $f_1$ and $f_2$, there exists an optimal matching $\gamma=(\gamma_1,\gamma_2)$ that is a union of P-segments and N-segments. We now outline our algorithm for producing such an optimal matching. Throughout this section, we continue using the notation developed in the previous section for our curves $f_1$ and $f_2$ and their SRVF's $q_1$ and $q_2$. We assume that the $q_i's$ are step functions that do not take the value zero on any of their subintervals.

The algorithm examines each vertex $(s_i,t_j)$, one row at at time, in the order 
$$(s_0,t_0), (s_1,t_0), (s_2,t_0),\dots,(s_0,t_1),(s_1,t_1), (s_2,t_1),\dots, (s_{m-1},t_n),(s_m,t_n)$$
When it arrives at a vertex $(s_i,t_j)$, it checks whether an optimal segment has been found from $(s_0,t_0)$ to $(s_i,t_j)$. If no such optimal segment has been found, it skips to the next vertex.

However, if such an optimal segment has been found, it implements a ``searchlight" procedure, looking for segments starting from $(s_i,t_j)$, as follows:
\begin{itemize}
\item If $W_{i+1,j+1}\leq 0$, the algorithm finds all possible N-segments beginning at $(s_i,t_j)$. Suppose such an N-segment ends at $(s_k,t_l)$. The algorithm checks whether the value of the optimal segment from $(s_0,t_0)$ to $(s_i,t_j)$ is higher than the value of the best segment found so far from $(s_0,t_0)$ to $(s_k,t_l)$. If it is, then the union of these two segments yields a new best possible segment from $(s_0,t_0)$ to $(s_k,t_l)$, and this segment is recorded as such. If it is not, then this N-segment is simply ignored.
\item If $W_{i+1,j+1}>0$, then the algorithm examines P-segments beginning at $(s_i,t_j)$. It does not have to examine all such P-segments, because of the slope restriction imposed by the last clause of Theorem \ref{MainTheorem}.  To be more precise, by considering the final slope of the last P-segment occurring in the optimal path from $(s_0,t_0)$ to $(s_i,t_j)$ and the value of four relevant weights, the last clause of Theorem \ref{MainTheorem} specifies an allowable range of slopes for the next P-segment. Our searchlight procedure examines all P-segments beginning at $(s_i,t_j)$ whose initial slopes are within this range. (We will soon give some more details on how we accomplish the enumeration of these P-segments.) Suppose such a P-segment ends at $(s_k,t_l)$. The algorithm checks whether the sum of the values of this new P-segment and the optimal segment from $(s_0,t_0)$ to $(s_i,t_j)$ is greater than the value of the best segment found so far from $(s_0,t_0)$ to $(s_k,t_l)$. If it is greater, the union of these two segments yields a new candidate for best possible segment from $(s_0,t_0)$ to $(s_k,t_l)$, and this segment is recorded as such. If it is not, then this new P-segment is ignored.

\end{itemize}
During the application of this algorithm, by the time we are examining a vertex $(s_i,t_j)$, we have already determined whether or not there exists a segment from $(s_0,t_0)$ to $(s_i,t_j)$ that follows the rules of Theorem \ref{MainTheorem}. 

Thus when we arrive at the final vertex $(s_m,t_n)$, we will have determined the best possible segment from $(s_0,t_0)$ to $(s_m,t_n)$.

We now make some further comments on the searchlight procedure alluded to above. In the first case, we are searching for all possible N-segments starting at $(s_i,t_j)$. This can be accomplished by a relatively simple combinatorial procedure, searching for vertices above and to the right of $(s_i,t_j)$ which will be the endpoint of an allowable N-segment. 

However, the searchlight procedure has more subtlety in the second case, where we are searching for all possible P-segments, with starting slope within a given interval, say $[h_1,h_2]$. Because of this subtlety, we give some more details about how this is accomplished. In order to make sure we don't miss any allowable P-segments due to round-off error, we begin by choosing an initial slope $h_1-\epsilon$, where $\epsilon$ denotes some convenient small positive number. Then, we construct a segment beginning at the vertex $(s_i,t_j)$ and following the slope-change rules from the definition of P-segment whenever we cross from one block to the next. There is essentially a zero probability that this segment will meet a vertex, so the segment ends when it arrives at either the vertical line $s=1$ or the horizontal line $t=1$. Technically, this segment is not a P-segment, because its final point is not a vertex. The idea of the searchlight algorithm is that we want to find the next initial slope above ($h_1-\epsilon$) that will result in a P-segment that actually terminates at a vertex. There is a nice trick for accomplishing this. Note that the slope of this segment changes each time it passes from one block to the next, because of the change in the weights as we pass from one block to the next. However, it is very easy to perform a PL reparameterization of the original curves, that will result in all the blocks that this path passes through having the same weight! For example, consider the case in  which the first edge-crossing of our segment takes it from $G_{i+1,j+1}$ to $G_{i+2,j+1}$. By choosing a linear reparameterization $\gamma:[s_{i+1},\tilde s_{i+2}]\to [s_{i+1},s_{i+2}]$, we can change the value of the $q$-function of $f_1$ on this portion of the curve to any multiple of its original value $u_{i+2}$ that we desire. Therefore, we can change the weight $W_{i+2,j+1}$ to make it equal to the weight $W_{i+1,j+1}$ by such a reparameterization. (Of course, we must translate the values of $s_{k}$ for all $k>i+2$ in order to accommodate the new value of $\tilde s_{i+2}$.). Since the weights of these two blocks are now equal, it follows from the slope transition formula that the slope of the segment will now remain the same as our segment basses from  $G_{i+1,j+1}$ to $G_{i+2,j+1}$. We proceed along our segment, making a similar reparametrization of either $f_1$ or $f_2$ each time the segment passes from one block to the next. The result of this procedure will be that our entire segment has the same slope (equal to its initial slope in $G_{i+1,j+1}$). Note that the total parameter intervals will no longer be the unit  intervals that they were to start with, but that doesn't matter. Also, note that the coordinates of several of the vertices will have been changed by these reparameterizations.

We need to find the lowest slope above $h_1-\epsilon$ for which the segment encounters a vertex. But, because the slopes are all the same along the segment, this becomes easy. Let $S$ denote the set of vertices that are either the upper end of a vertical edge crossed by our segment, or the left end of a horizontal edge crossed by our segment. For each of the vertices $(s_k,t_l)$ in $S$, compute the ratio $t_l/s_k$; the lowest value of this ratio will obviously be the lowest initial slope above $h_1-\epsilon$ for which our segment encounters a vertex. Call this new slope $\tilde h_1$. Going back to our original parameterizations, we have our first P-segment, starting at $(s_i,t_j)$, with initial slope $\tilde h_1$. 

To find the next P-segment, we begin by constructing a segment starting at $(s_i,t_j)$, with slope $\tilde h_1+\epsilon$ for a very small $\epsilon$, that follows the slope-change rules whenever it passes from one block to another. There is a zero probability that this segment will encounter a vertex, so it will end when it arrives at either the vertical line $s=1$ or the horizontal line $t=1$. To find the next slope above $\tilde h_1+\epsilon$ that will yield a P-segment, we use exactly the same path straightening procedure that we just described. We proceed in this manner until we arrive at a slope above $h_2$. This gives us all the P-segments starting at $(s_i,t_j)$ with initial slopes in the required range. Note that for each P-segment we construct, we just need to construct one ``test" segment to find it.

\section{Examples}

In the following pages, we present the results produced by implementing the aforementioned algorithm on different pairs of $1D$, $2D$ and $3D$-curves. In case of $1D$-curves, $A$ shows the original curves as graphs, $B$ shows the aligned curves and $C$ shows the optimal matching on $I \times I$ grid. In case of $2D$ and $3D$-curves, the alignment of the curves is shown in figure $A$, the geodesic is shown in $B$ and $C$ represents the optimal matching on the $I \times I$ grid for the pair of curves. The following table shows the list of the pairs of curves.

\begin{table}[h]
\begin{center}
\begin{tabular}{ c | c }
\hline \\
EX & DESCRIPTION OF THE PAIRS OF CURVES  \\ [0.8ex]
\hline 
\hline \\
$1(1D)$ & $f_{1}(t) = f(t)$, $f_{2}(t) = g(t)$, $t \in \left\{\frac{n}{5} \right\}_{n=0}^{5}$, taken from a random data set\\[1.5ex]
$2(1D)$ & $f_{1}(t) = f(t)$, $f_{2}(t) = g(t)$, $t \in \left\{\frac{n}{100} \right\}_{n=0}^{100}$, taken from a simulated data set\\[1.5ex]
$3(2D)$ & $f_{1}(t) = (t,f(t))$, $f_{2}(t) = (t,g(t))$, $t \in \left\{\frac{n}{45} \right\}_{n=0}^{45}$, taken from the female growth data set \cite{TuddenhamEtAl} \\[1.5ex]
$4(2D)$ & $f_{1}(t) = (2 \pi t , \ 2 \pi t)$ and $f_{2}(t) = (2 \pi t , \ \sin(6 \pi t))$, $t \in \left\{\frac{n}{45} \right\}_{n=0}^{45}$\\[1.5ex]
$5(2D)$ & $f_{1}(t) = (1+\cos(2 \pi (1-t)) , \sin (2 \pi (1-t)))$, $f_{2}(t) = (-1 + \cos(-2 \pi t) , \sin(2 \pi t))$ , $t \in \left\{\frac{n}{45} \right\}_{n=0}^{45}$\\[1.5ex]
$6(2D)$ & $f_{1}(t) = (1+\cos(2 \pi (1-t)) , \sin (2 \pi (1-t)))$, $f_{2}(t) = (-1 + \cos(-2 \pi t) , \sin(2 \pi t))$ , $t \in \left\{\frac{n}{3} \right\}_{n=0}^{3}$\\[1.5ex]
$7(2D)$ & $f_{1}(t) = (2 \pi t , \ \sin(6 \pi t))$, $f_{2}(t) = (2 \pi t , \ \sin(4 \pi t))$ , $t \in \left\{\frac{n}{45} \right\}_{n=0}^{45}$\\[1.5ex]
$8(3D)$ & $f_{1}(t) = (\cos 4\pi t , \ \sin 4\pi t , \ t)$, $f_{2}(t) = (\cos 8\pi t , \ \sin 8\pi t , \ t)$ , $t \in \left\{\frac{n}{50} \right\}_{n=0}^{50}$\\[1.5ex]
$9(3D)$ & $f_{1}(t) = (4\pi t\cos (4\pi t) , 4\pi t\sin (4\pi t) , (4\pi t)^{2})$, $f_{2}(t) = (4\pi t\cos (4\pi t) , -4\pi t\sin(4\pi t) , (4\pi t)^{2})$ , $t \in \left\{\frac{n}{50} \right\}_{n=0}^{50}$\\[1.5ex]
\hline
\end{tabular}
\end{center}
\end{table}

\begin{figure}[h]
\begin{center}
\includegraphics[scale=0.5]{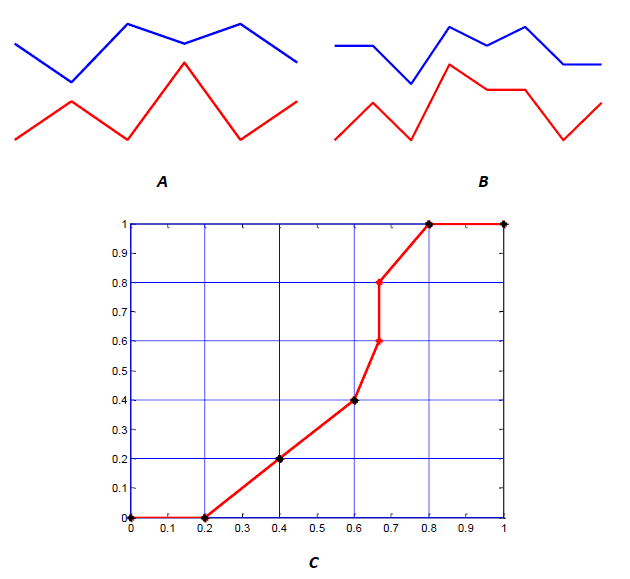}
\caption{Example $1(1D)$. Distance before alignment is $1.4815$. Distance after alignment is $0.5071$.}
\label{ex1}
\end{center}
\end{figure}

\begin{figure}[h]
\begin{center}
\includegraphics[scale=0.55]{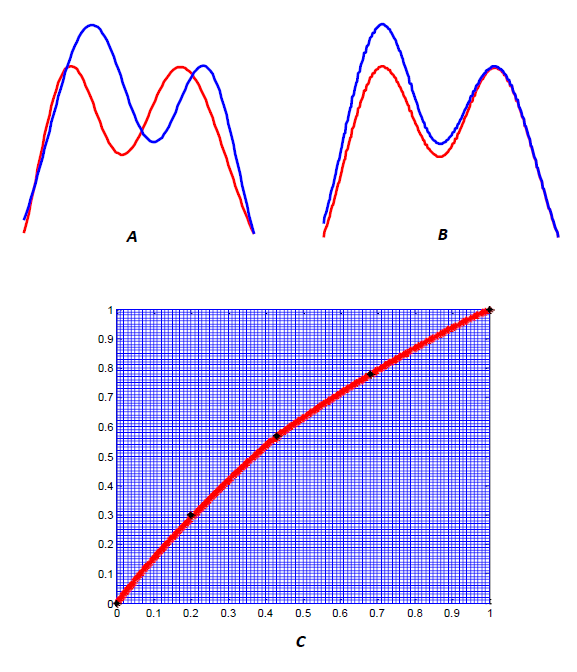}
\caption{Example $2(1D)$. Distance before alignment is $1.4312$. Distance after alignment is $0.1195$.}
\label{ex2}
\end{center}
\end{figure}
  
\begin{figure}[h]
\begin{center}
\includegraphics[scale=0.45]{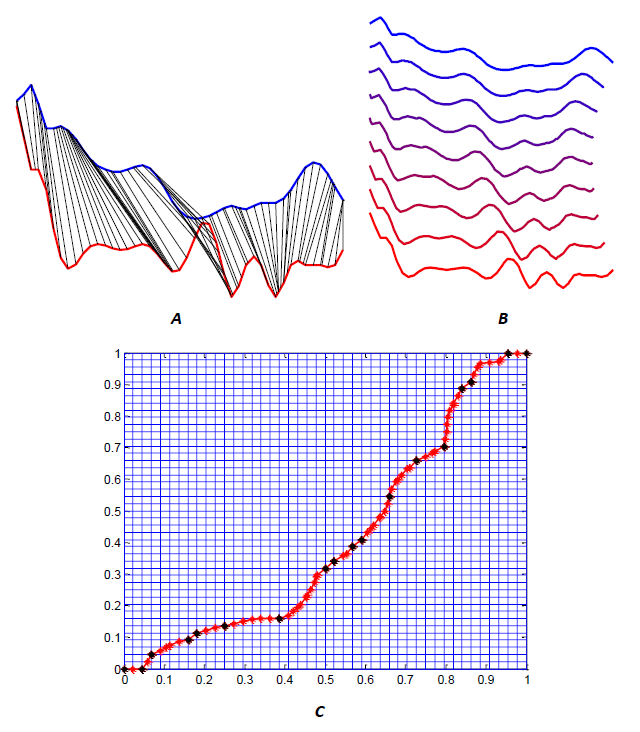}
\caption{Example $3(2D)$. Distance before alignment is $7.0108$. Distance after alignment is $4.0721$.}
\label{ex3}
\end{center}
\end{figure}

\begin{figure}[h]
\begin{center}
\includegraphics[scale=0.45]{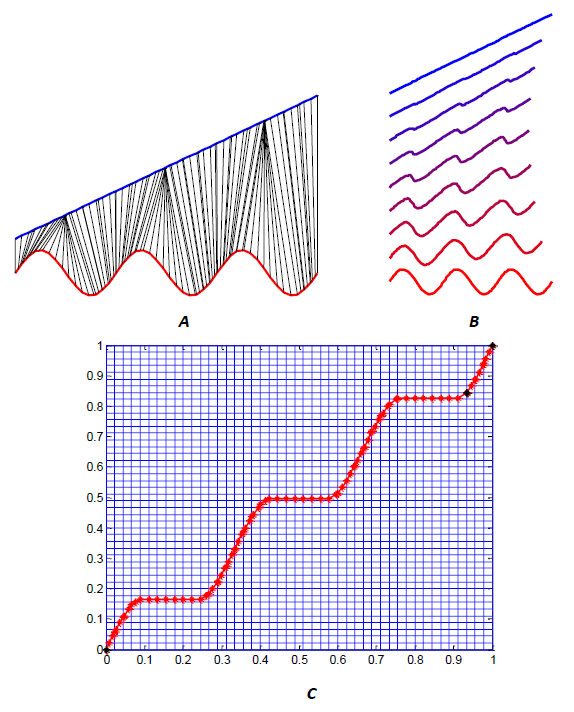}
\caption{Example $4(2D)$. Distance before alignment is $3.9107$. Distance after alignment is $2.8418$.}
\label{ex4}
\end{center}
\end{figure}

\begin{figure}[h]
\begin{center}
\includegraphics[scale=0.45]{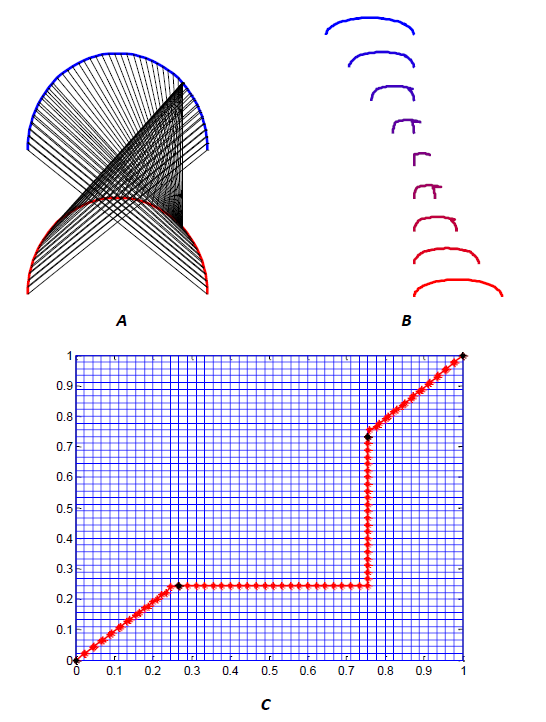}
\caption{Example $5(2D)$. Distance before alignment is $2.5064$. Distance after alignment is $2.0683$.}
\label{ex5}
\end{center}
\end{figure}

\begin{figure}[h]
\begin{center}
\includegraphics[scale=0.45]{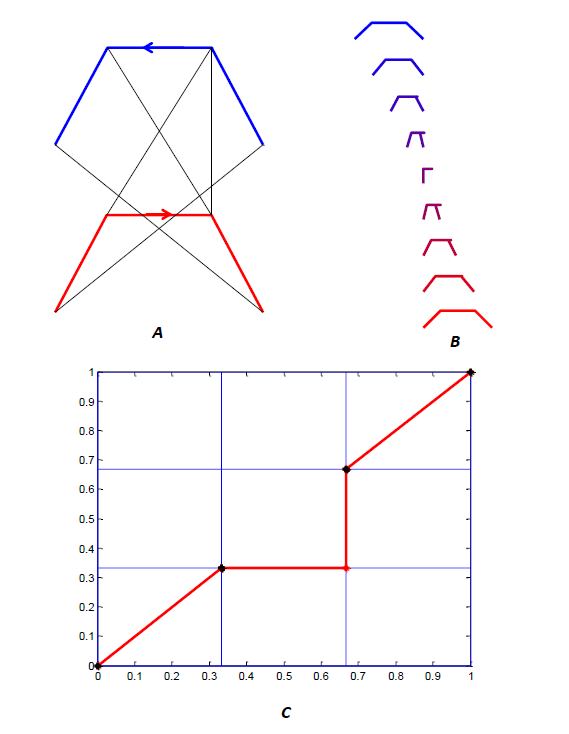}
\caption{Example $6(2D)$. Distance before alignment is $2.4495$. Distance after alignment is $2$.}
\label{ex6}
\end{center}
\end{figure}

\begin{figure}[h]
\begin{center}
\includegraphics[scale=0.45]{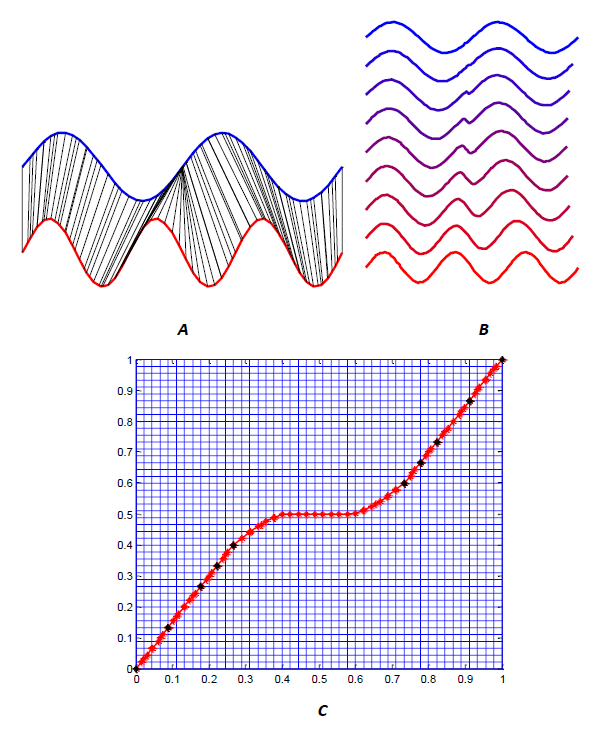}
\caption{Example $7(2D)$. Distance before alignment is $4.1655$. Distance after alignment is $1.7899$.}
\label{ex7}
\end{center}
\end{figure}

\begin{figure}[h]
\begin{center}
\includegraphics[scale=0.45]{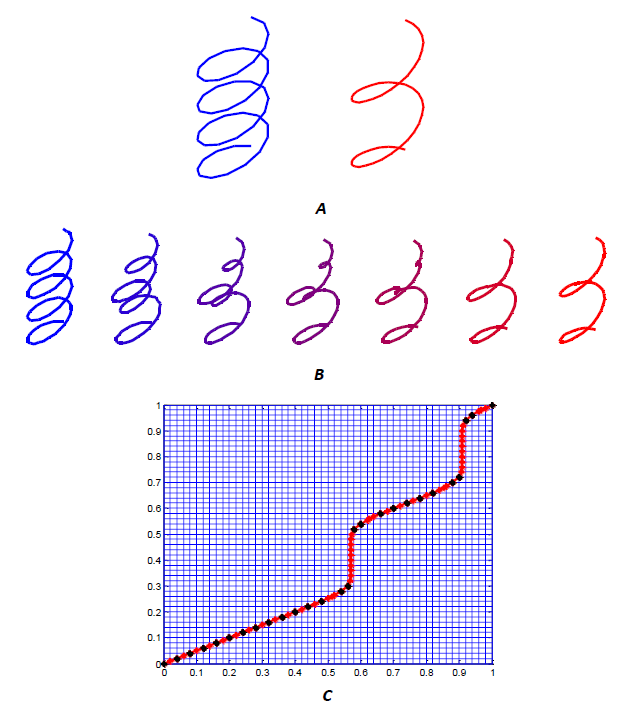}
\caption{Example $8(3D)$. Distance before alignment is $6.1114$. Distance after alignment is $3.2117$.}
\label{ex8}
\end{center}
\end{figure}

\begin{figure}[h]
\begin{center}
\includegraphics[scale=0.45]{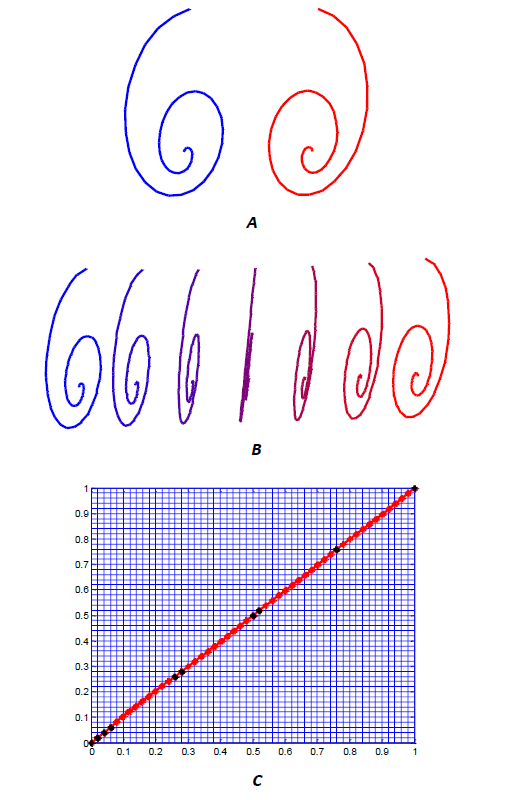}
\caption{Example $9(3D)$. Distance before alignment is $8.5302$. Distance after alignment is $8.5253$.}
\label{ex9}
\end{center}
\end{figure}

\begin{figure}[h]
\begin{center}
\includegraphics[scale=0.45]{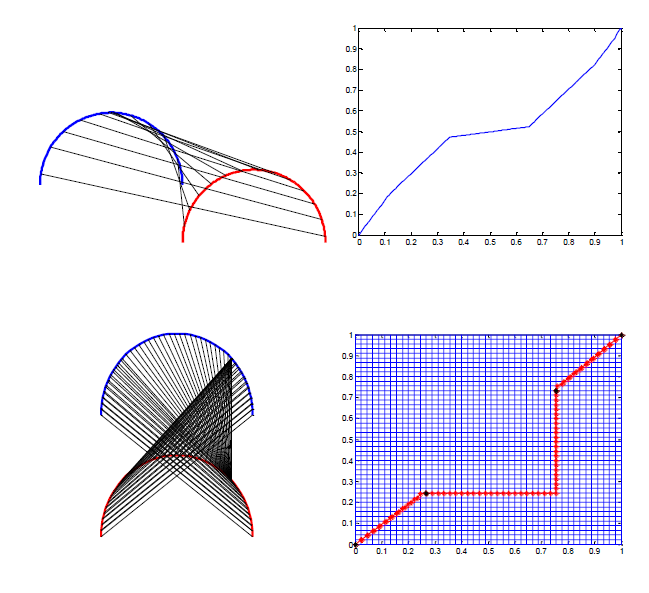}
\caption{Comparing results from the Dynamic programming (top row) with the algorithm (bottom row). Distance before alignment is $1.57$. Distance after alignment : $1.5239$ (using DP) ; $1.2457$ (using our algorithm).}
\label{ex10}
\end{center}
\end{figure}

\clearpage
\bibliographystyle{plain}
\bibliography{fa_match}

\end{document}